\documentclass[11pt]{article}
\usepackage{amsmath,amsfonts,amssymb,amsthm}
\font\Bbb=msbm10
\def\R{\hbox{\Bbb R}}

\newcount\diagramnumber

\def\address#1
{
	\expandafter\def\expandafter\@aabuffer\expandafter
	{\@aabuffer{\affiliationfont{#1}}\relax\par
	\vspace*{13pt}}
}

\newtheorem{theorem}{Theorem}[section]
\newtheorem{lemma}[theorem]{Lemma}
\newtheorem{proposition}[theorem]{Proposition}
\newtheorem{corollary}[theorem]{Corollary}
\newtheorem{definition}[theorem]{Definition}
\newtheorem{remark}[theorem]{Remark}
\newtheorem{example}[theorem]{Example}

\def\prbox
{
	\par
	\vskip-\lastskip\vskip-\baselineskip\hbox to \hsize{\hfill\fboxsep0pt\fbox{\phantom{\vrule width5pt height5pt depth0pt}}}
}

\usepackage{epsfig}

\begin{document}

%\markboth{Bartholomew, Fenn, Kamada, Kamada}
%{Doodles on Surfaces}

\title{Doodles on Surfaces}

\author{Andrew Bartholomew \\ 
School of Mathematical Sciences, University of Sussex\\
Falmer, Brighton, BN1 9RH, England\\
e-mail address: andrewb@layer8.co.uk \\ 
\\
Roger Fenn \\ 
School of Mathematical Sciences, University of Sussex\\
Falmer, Brighton, BN1 9RH, England\\
e-mail address: rogerf@sussex.ac.uk \\ 
\\ 
Naoko Kamada \\ 
Graduate School of Natural Sciences, Nagoya City University\\
Nagoya, Aichi 467-8501, Japan\\
e-mail address: kamada@nsc.nagoya-cu.ac.jp \\ 
\\ 
Seiichi Kamada \\ 
Department of Mathematics, Osaka City University\\
Osaka, Osaka 558-8585, Japan\\
e-mail address: skamada@sci.osaka-cu.ac.jp
}

\maketitle

\begin{abstract}
Doodles were introduced in \cite{FT} but were restricted to embedded circles in the $2$-sphere. Khovanov, \cite{MK}, extended the idea to immersed circles in the $2$-sphere. In this paper we further extend the range of doodles to any closed oriented surface.  
Uniqueness of minimal representatives is proved, and various example of doodles are given with their minimal representatives. 
We also introduce the notion of virtual doodles, and show that there is a natural one-to-one correspondence between doodles on surfaces and virtual doodles on the plane.  
\end{abstract}

{\bf keywords:} {doodles, virtual doodles, minimal diagrams, immersed circles} 


\section{Introduction}

Doodles were first introduced by the second author and Taylor in \cite{FT}. The original definition of a doodle was a collection of embedded circles in the $2$-sphere $S^2$ with no triple or higher intersection points.
Khovanov, \cite{MK}, extended the idea to allow each component to be an immersed circle in $S^2$. 
Further references on plane curves are \cite{A1} and \cite{A2}. 

In this paper, we further extend the range of doodles to immersed circles in closed oriented surfaces of any genus.  
Doodles on surfaces can be regarded as equivalence classes of generically immersed circles in closed oriented surfaces, called regular representatives or diagrams, 
under an equivalence relation.  The equivalence relation is generated by introducing/removing locally a monogon or bigon and 
surgeries on ambient surfaces avoiding the diagrams (Section~\ref{sect:defs}). 

A doodle diagram on a surface is called minimal if the interior of every region is 
simply connected and there are no monogons and bigons.   
In Section~\ref{sect:uniqueness}, we prove uniqueneness of minimal diagrams, which means that 
for each doodle, there is a unique minimal diagram representing the doodle (Theorem~\ref{thm:minimal}).  
It is analogous to Kuperberg's theorem, \cite{GK}, in virtual knot theory.  

To show this theorem, we generalise Newman's simplification procedure, \cite{N}. 
The main result of Section~\ref{sect:uniqueness} is Theorem~\ref{thm:proofreductiongraph}:  
The graph $\cal D$ of doodle diagrams with levels is a \lq proof reduction graph\rq.  
Theorem~\ref{thm:proofreductiongraph} implies that the uniqueness theorem (Theorem~\ref{thm:minimal}) and 
a theorem on characterization of minimal diagrams (Theorem~\ref{thm:characterize}), which claims that a diagram is minimal if and only if it is a 
 diagram with minimum crossing number and the ambient surface has minimum genus and  maximum component number.  

In Section~\ref{sect:example}, we observe minimal doodle diagrams and provide examples of planar doodles by giving sequences of minimal planar diagrams.  

The notion of a virtual doodle is introduced in Section~\ref{sect:virtual}, and it is proved in Section~\ref{sect:virtualsame} that 
there is a natural one-to-one correspondence between virtual doodles and doodles on surfaces (Theorem~\ref{thm:bijection}).  This correspondence is analogous to the correspondence between virtual links and stable equivalence classes of link diagrams on surfaces in \cite{CKS, KK}.

Section~\ref{sect:minimalvirtual} is devoted to showing examples of minimal virtual doodle diagrams and an observation on a relationship between the genera of doodles and virtual crossings of virtual doodles.   

%In a later paper we will consider doodle braids, biquandle invariants, cobordism of doodles, the $\mu$ invariant, Khovanov group and commutator identities.

It is shown in \cite{Fe}
that planar doodles induce commutator identities in the free group, \cite{H}.
Commutator identities related to doodles on surfaces will be discussed in a later paper.
For a unified treatment of generalized knot theories see \cite{F}.  
%We will restrict ourselves to orientable surfaces. The treatment of non-orientable surfaces will have to wait till a later paper. 

The authors would like to express their thanks to Victoria Lebed for pointing out that there was a duplication of the list of 
minimal virtual doodles with $4$ crossings in the previous version of this paper posted on arXiv:1612.08473v1 (see Remark~\ref{Lebed}).    
This work was supported by JSPS KAKENHI Grant Numbers 26287013 and 15K04879.

\section{Definitions} \label{sect:defs}

A {\bf doodle} is represented by a map  $f: \coprod_i S^1_i\to \Sigma $
from $n$ disjoint circles to a closed oriented surface $\Sigma$ so that
$|f^{-1}f(x)|<3\hbox{ for all } x\in \coprod_i S^1_i $.
That is, no triple or higher multiple points are created.  
To avoid \lq\lq wild\rq\rq doodles, we further assume that $f$ is a smooth map with a (topological) normal bundle. 

The map $f$ restricted to any one circle is called a {\bf component}. 

Two representatives are said to be {\bf equivalent} if they are equivalent under the equivalence generated by 1. 2. and  3. defined as follows. 

\begin{enumerate}
\item Homeomorphic equivalence,
\item Homotopy through doodle representatives,
\item Surface surgery disjoint from the diagram.
\end{enumerate}

The equivalence class is called a {\bf doodle}, or a {\bf doodle on a surface}. However, as is the usual custom, we will often not distinguish between a doodle and its representative.

%If the doodles are equivalent under the equivalence generated by 1. and 2. they are called {\bf homeotopic} and the equivalence is called a {\bf homeotopy}.

A doodle representative is {\bf regular} if it is an immersion whose multiple points 
are a finite number of transverse crossings. Clearly any doodle can be regularly represented. 
The image of a regular representative is called a {\bf diagram} on the surface.  
A pair $(\Sigma, D)$ of a diagram $D$ on a surface $\Sigma$ is also called a diagram on a surface. 

A representative $f: \coprod_i S^1_i\to \Sigma $ or a diagram is called {\bf admissible} if each connected component of the surface $\Sigma$ intersects with the image of $f$.  Throughout this paper we always assume that representatives and diagrams are admissible unless otherwise stated.  

A doodle is called {\bf planar} if there is a representative on the $2$-sphere. Although such a representative is often depicted on the plane, one should consider it on the $2$-sphere.

Figure~\ref{fig01Poppyfig02Borromean} shows diagrams of two planar doodles. The first, called the {\bf poppy}, has one component and the other with 3 components is called the {\bf Borromean} doodle. 

%\diagram{The Poppy}
%\diagram{The Borromean Doodle}
    \begin{figure}[h]
    \centerline{\epsfig{file=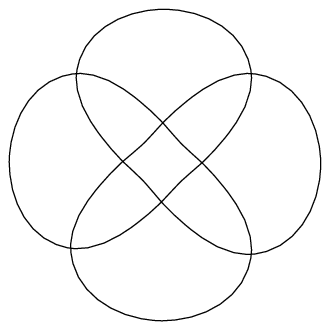, height=2.5cm} 
    \qquad \qquad 
    \epsfig{file=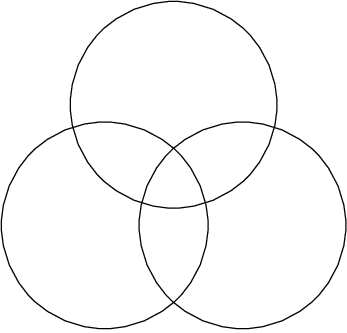, height=2.5cm}}
    \vspace*{8pt}
    \caption{The Poppy and the Borromean Doodle}\label{fig01Poppyfig02Borromean}
    \end{figure}

These are the first members of an infinite family of planar doodles considered later in the paper.

The {\bf Hopf doodle} is represented by the longitude and meridian of a torus. That is $S^1\times *\cup  *\times S^1 \subset S^1\times S^1$.

1. Homeomorphic equivalence means that the doodles are \lq\lq topologically the same\rq\rq.  So for two doodles $f$ and $g$ with the same number of components there are homeomorphisms $s:\coprod_i S^1_i:\to \coprod_i S^1_i$ and $t:\Sigma\to \Sigma$ which make the square

$$\begin{array}{ccc}
\coprod_i S^1_i&\buildrel f \over \longrightarrow& \Sigma \\
\downarrow s&&t\downarrow \\
\coprod_i S^1_i&\buildrel g \over \longrightarrow& \Sigma
\end{array}$$
commute.  
Here we assume that $t$ respects the orientation of $\Sigma$.  
When we consider {\bf oriented} (or {\bf unoriented}) doodles, it is (or is not) required that the homeomorphism $s:\coprod_i S^1_i:\to \coprod_i S^1_i$ respects the standard orientation of the circles.  
When we consider {\bf ordered} (or {\bf unordered}) doodles, it is (or is not) required that the homeomorphism $s:\coprod_i S^1_i:\to \coprod_i S^1_i$ respects the indices.  

2. Any homotopy of regular doodles can be assumed to be divided up into Reidemeister type moves:
$H_1^+$ which generates a curl, $H_1^-$ which deletes a curl, $H_2^+$ which generates a bigon and $H_2^-$ which deletes it. See Figure~\ref{fig03H_1pm1fig04H_2pm1}. 

%\diagram{$H_1^{\pm 1}$}
%\diagram{$H_2^{\pm 1}$}
    \begin{figure}[h]
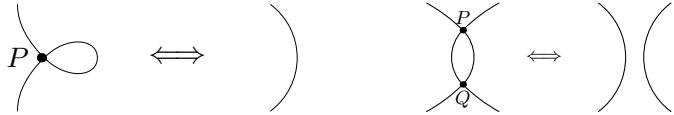

    \centerline{\epsfig{file=doodlefig03H1pm1.eps, height=1.5cm} 
    \qquad \qquad 
    \epsfig{file=doodlefig04H2pm1.eps, height=1.5cm} }
    \vspace*{8pt}
    \caption{$H_1^{\pm 1}$ and $H_2^{\pm 1}$}\label{fig03H_1pm1fig04H_2pm1}
    \end{figure}

Under $H_1^+$ a self crossing point $P$ and a monogon with empty interior are introduced and under $H_1^-$ both are eliminated. 

Under $H_2^+$ two crossing point $P, Q$ are introduced and a bigon with empty interior between them. The crossing points  may or may not be crossings of different components. Under $H_2^-$ the crossing points and  the bigon are eliminated.

3. Surface surgery can be divided into a sequence of handle additions, $h^+$, and handle eliminations, $h^-$. These operations have to be disjoint from the doodle. Handle addition can be described as follows. Let $D_0, D_1$ be disjoint closed discs in the surface disjoint from the diagram. Remove the interiors of the discs and add an annulus, $A=S^1\times [0,1]$, by its boundary to the boundary of the discs. The circle $b=S^1\times 1/2$ is called the {\bf belt} of the handle. See Figure~\ref{fig05ahandle}. 

%\diagram{A Handle}
    \begin{figure}[h]
    \centerline{\epsfig{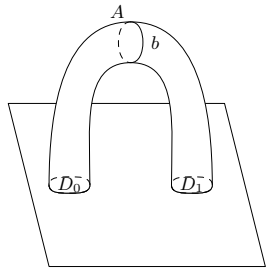}}
    \vspace*{8pt}
    \caption{A Handle}\label{fig05ahandle}
    \end{figure}

The elimination of a handle is the reverse of this procedure. Let $b$, (the belt of a handle),  be a simple closed curve in the surface disjoint from the doodle. Then a regular neighbourhood, $A$ of $b$ is homeomorphic to an annulus and can be chosen disjoint from the doodle. Remove the interior of the annulus and glue two discs via their boundary circles to the boundary of the annulus.  Since we consider admissible diagrams, we avoid handle eliminations which make the diagrams inadmissible. 

%\section{Prime Doodles}
%Let $D$ be a doodle. We say that $D$ is {\bf separated} if it has the maximum possible number of topological components allowed under the doodle moves. It is not difficult to see that if $D_1$ and $D_2$ are separated and equivalent under the doodle moves then they are equivalent under moves between separated doodles.
%
%If the maximum number of topological components is one then we call the doodle {\bf prime}. 

\section{Uniqueness of Minimal Doodle Diagrams} \label{sect:uniqueness}

In this section we see that every doodle has a unique representative doodle diagram which is minimal with respect to the number of crossing points and the genus of the underlying surface. To do this we first formalise a procedure which is common in mathematical proofs and was introduced by Newman, \cite{N}. It has since been recently considered by Matveev, \cite{M}  and Bergman, \cite{B}.

Let $G$ be a graph. We will denote the vertices of $G$ by capital Roman letters and the edges by their end points: for example $e=KL$. 

As usual, a {\bf path} in $G$ from $A$ to $B$ is a sequence of vertices
$$A=K_0, K_1, K_2, \cdots, K_{n-1}, K_n=B$$
in which $K_{i-1}K_i$ is an edge, $i=1,\ldots,n$.

A path is called {\bf simple} if it has no re-entrant vertices, $K_i=K_j$ where $0<i<j<n$. Any path can be replaced by a simple path with the same end points.

We will call $G$ a {\bf graph with levels} if each vertex, $K$, of $G$ has a level, $|K|$, which is an element of a totally ordered set, called the {\bf level set}, and that any two end vertices of an edge have different levels. This defines an ordering on the edges of $G$. If an edge $e$ of $G$ has vertices $K$ and $L$ and $|K|>|L|$ then $e$ is oriented from $K$ to $L$.  We will write an edge oriented in this manner as $e=K\searrow L$ or $L\swarrow K$ and picture $K$ on the page as being above $L$. We say that $K$ {\bf collapses} to $L$ or $L$ {\bf expands} to $K$.

A path of the form
$$A\searrow K_1\searrow K_2\cdots \searrow K_{n-1}\searrow B$$
is called a {\bf descending} path. The inverse of a descending path is called an {\bf ascending} path.

The first condition we impose on $G$ is the {\bf finite descending path property}.

{\bf FDPP:} There are no infinite descending paths.

A {\bf root} of $G$ is a sink. That is a vertex, $R$,  with no outgoing edges, $R\searrow L$. So a root is a local minimum.

\begin{lemma}If a graph with levels has the FDPP property then every vertex is either a root or is connected to a root by a descending path.
\end{lemma}

\begin{proof} 
 If $K$ is not a root then it has a descending edge, $K\searrow K_1$. If $K_1$ is not a root it has a descending edge, $K_1\searrow K_2$ and so on. By the  FDPP property this process must terminate after a finite number of steps with a root. 
 \end{proof}

From now on we will assume that all graphs have the FDPP. Having this condition, which guarantees the existence of roots, we now look at conditions which make the root unique.

To this end we consider the {\bf unique root property} and the {\bf diamond condition}.

{\bf URP:} Every pair of vertices in a path component descend to a unique root.

To define the diamond condition we need a few definitions.

A {\bf peak, (valley)} is an ascending (descending) path composed with a descending (ascending) path. A peak (valley) is {\bf simple} if it consists of just two edges. 

{\bf DC:}  Any peak can be replaced by a valley with the same end points. In particular if $U$ descends to $X$ and $Y$ then either $X=Y$ or there is a vertex $V$ which ascends to $X$ and $Y$.
\def\addots{\mathinner{\mkern1mu\raise1pt\vbox{\kern7pt\hbox{.}}\mkern2mu
\raise4pt\hbox{.}\mkern2mu\raise7pt\hbox{.}\mkern1mu}}

$$\begin{array}{ccccccccc}
&&&&U&&&&\\
&&&\swarrow&&\searrow&&&\\
&&\addots&&&&\ddots&&\\
&\swarrow&&&&&&\searrow&\\
X&&&&&&&&Y\\
&\searrow&&&&&&\swarrow&\\
&&\ddots&&&&\addots&&\\
&&&\searrow&&\swarrow&&&\\
&&&&V&&&&\\
\end{array}$$

\begin{lemma}URP and DC are equivalent conditions.
\end{lemma}

\begin{proof} 
 Suppose a graph has the URP and $U$ descends to $X$ and $Y$. Then $X$ and $Y$ are clearly in the same path component and so both descend to a common root, $V$ say.

Conversely suppose the graph has the DC and $X$ descends to different roots $R_1$ and $R_2$. Then unless $R_1=R_2$ this contradicts the DC. So every vertex descends to a unique root.

Now suppose that $X$ and $Y$ are joined by a path 
$$X=K_0, K_1,  K_2, \cdots, K_{n-1}, K_n=Y$$
 and yet descend to different roots $R_1$ and $R_2$. Then somewhere in the path, vertices $K_i$ and $K_{i+1}$ descend to different roots. We may as well assume that $K_i\searrow K_{i+1}$. Then $K_i$ descends to one root and via $K_{i+1}$ to another, contradicting the above.
\end{proof}

We say that a graph with levels is a {\bf proof reduction} graph if it has both the FDPP and the DC/URP properties.
By the above, all path components of a proof reduction graph have a unique root. This is clearly a useful property but we need practical methods to recognise such a graph. We do this by localising the DC as follows.

{\bf LDC:} A graph has the {\bf local diamond condition} if given a simple peak $X\swarrow U\searrow Y$ there is a path, $X=K_0, K_1, K_2, \cdots, K_{n-1}, K_n=Y$, from $X$ to $Y$ such that if the path contains a simple peak $K_{i-1}\swarrow K_i\searrow K_{i+1}$ then there is an edge $U\searrow K_i$.

\begin{lemma}The local diamond condition, LDC and the diamond condition, DC are equivalent.
\end{lemma}

\begin{proof} 
Clearly DC implies LDC because a valley does not have a peak. 

Now consider a graph with the LDC.  Because of FDPP every vertex which is not a root is connected to at least one root by a  descending path. Our task is to show that this root is unique.

Let us call a vertex {\bf regular} if it  is connected to a unique root by a descending path: otherwise we call it {\bf irregular}. Clearly regular vertices exist. A root is an example. Our task is to show that irregular vertices do not exist. We will assume the contrary and obtain a contradiction.

In that case there must be an irregular vertex $L$ such that for every edge $L\searrow K$ the vertex $K$ is regular. If not we could construct an infinite descending path of irregular vertices. Indeed, we will chose $L$ so that every descending path from $L$ must consist of regular vertices apart from $L$.

For such an irregular vertex $L$ there is a simple peak $X\swarrow L\searrow Y$ such that $X$, $Y$ descend to unique but different roots. Chose $X$ and $Y$ so that the path, $X=K_0, K_1, K_2, \cdots, K_{n-1}, K_n=Y$ predicted by the LDC has shortest possible length.

The hypothesis of the LDC means that if the path joining $X$ to $Y$  has a simple peak $K_{i-1}\swarrow K_i\searrow K_{i+1}$ as a subpath then there is an edge $L\searrow K_i$. This means that $K_i$ is regular and so  descends to a unique root. 
%In particular $K_{i-1}$ and $K_{i+1}$ must descend to the same unique root.
%$So either the pair $X, K_{i+1}$ or the pair $K_{i-1}, Y$ have shorter hypothetical paths.

This root must be different from one of the distinct roots of the pair  $X$, $Y$. So either the pair $X, K_i$ or the pair $K_i, Y$ has a shorter path.

So the path joining $X$ to $Y$ has no simple peaks and is therefore either a) ascending, b) descending or c) a valley. If a) or b) then $X$, $Y$ have a common root. If c) and  If $V$ is the base of this valley then $V$ has a unique root which must be the same for $X$ and $Y$.

Therefore irregular vertices cannot exist and all vertices are regular and connected to a unique root by a descending path.  Hence the URP is satisfied which implies the LDC. ~ 
\end{proof}

We now consider examples of proof reduction graphs.

{\bf Free Groups} The vertices are words in the symbols $x\in X$ and $x^{-1}\in X^{-1}$. The level of a word is its length. The expansions are insertions of pairs $xx^{-1}$ or $x^{-1}x$ in the words. It is easy to see that this has the local diamond condition. Hence every word  is equivalent to a unique reduced word. 

{\bf The singular braid monoid embeds in a group},  \cite{FKR}
Here the vertices are singular braids up to braid equivalence. The expansions are the introduction of pairs of cancelling singular crossings. The levels are the number of singular crossings.

The graph, $\cal D$, we are interested in has vertices consisting of (admissible and ordered) oriented (or unoriented) 
doodle diagrams on surfaces,  
$S=(\Sigma, D)$.  We assume that homeomorphic diagrams are the same vertex of $\cal D$.  
The level of a vertex is the number of crossings minus the Euler characteristic of the surface.  
The moves $H_1^{\pm1}, H_2^{\pm1}$ and  $h^{\pm1}$ raise or lower the level and correspond to the edges of the graph. 

%These are equivalent up to topological equivalence and a move called a {\bf floating move} which we now describe. A {\bf floating component} is a component which is the boundary of a disc disjoint from the rest of the diagram. A floating move picks up this disc and places it anywhere else on the same connected component of the surface disjoint from the rest of the diagram. 

\begin{theorem}\label{thm:proofreductiongraph}
The graph, $\cal D$,  of oriented (or unoriented) doodle diagrams with moves and levels defined above is a proof reduction graph.
\end{theorem}

Before proving this theorem, we prepare some terminology and a lemma. 

A {\bf trivial doodle diagram with one component} is a doodle diagram such that the diagram is a simple closed curve and the surface is a $2$-sphere.  A {\bf trivial doodle diagram with $n$ components} for $n \geq 1$ is a doodle diagram 
which is the disjoint union of $n$ trivial doodle diagrams with one component.  

A {\bf floating component} of a doodle diagram is a component which bounds a disc in the surface disjoint from the rest of the diagram. 

Note that we do not call a simple closed curve which bounds a disc in the ambient surface a trivial doodle diagram unless the surface is a $2$-sphere.  For example, let $(\Sigma, C)$ be a doodle diagram such that $\Sigma$ is a torus and $C$ is a simple closed curve which bounds a disc in the torus.  Then $(\Sigma, C)$ is not a trivial doodle diagram in our sense. We call $C$ a floating component, not a trivial diagram. 

\begin{lemma}\label{lem:floating}
Let $C$ be a  floating component of a doodle diagram $S= (\Sigma, D)$ and let 
$\Sigma_0$ be the connected component of $\Sigma$ containing $C$.     
If $(\Sigma_0, C)$  is not a trivial doodle diagram, then there exists a doodle diagram $S'$ such that 
$S \searrow S'$.  
\end{lemma}

\begin{proof}
We consider two cases: 
\begin{itemize}
\item[(1)] $(D \setminus C) \cap \Sigma_0 \neq \emptyset$, i.e., there exists a component of $D \setminus C$ on $\Sigma_0$. 
\item[(2)] $(D \setminus C) \cap \Sigma_0 = \emptyset$, i.e., the rest of the diagram $D \setminus C$ misses $\Sigma_0$. 
\end{itemize}

In case 1), apply an $h^{-1}$ move along a simple closed curve surrounding $C$ and we obtain an (admissible) doodle diagram which is the  disjoint union of $(\Sigma, D \setminus C)$ and a trivial doodle diagram with one component.  

In case 2), the surface $\Sigma_0$ has positive genus. We can apply an $h^{-1}$ move along a simple closed curve on $\Sigma_0$ avoiding $C$ to reduce the genus. 
\end{proof}

\begin{proof} (Proof of Theorem~\ref{thm:proofreductiongraph}) 
First we show that $\cal D$ has the FDPP property. 
Suppose that there is an infinite descending path and let $S=(\Sigma, D)$ be a vertex on the path. 
The number of crossings of $D$ is an upper bound of the number of $H_1^{-1}$ and $H_2^{-1}$ moves appearing on the path.  Since we are considering admissible doodle diagrams, the topology of $\Sigma$ with the number of circles of $D$ makes an  bound of the numbers of $h^{\pm1}$ moves appearing on the path.  This contradicts to that the path is infinite.  Thus $\cal D$ has the FDPP property. 

We will show that $\cal D$ has the DC or LDC property by looking at the possible cases.

Suppose $S_1,S_2,S_3$ is a sequence of diagrams created by $H_1^{\pm 1}, H_2^{\pm 1}$ moves such that 
$S_1\swarrow S_2\searrow S_3$. The first move creates a region $R$ and the second destroys a region $R'$. If $R=R'$ then the two moves cancel and $S_1$ is the same as $S_3$. The other cases are illustrated in 
Figure~\ref{fig18regions}. 

%\diagram{The regions $R$ and $R'$}
    \begin{figure}[h]
    \centerline{\epsfig{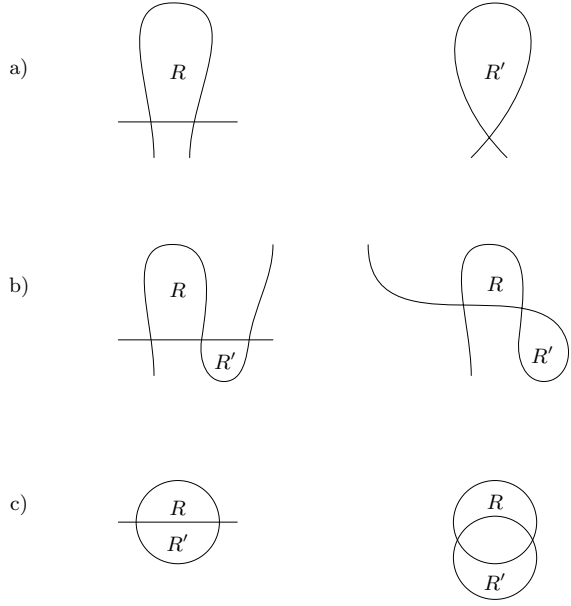}}
    \vspace*{8pt}
    \caption{The regions $R$ and $R'$}\label{fig18regions}
    \end{figure}

In case a) the regions $R$ and $R'$ are disjoint so by reversing the order of the moves there is an intermediate $S_2'$ such that $S_1\searrow S'_2\swarrow S_3$. The figure illustrates an $H_2^+$ and an $H_1^-$. 

In case b) the regions $R$ and $R'$ have a point in common. The region  $R$ is created by an $H_2^+$ move and the region  $R'$ is created accidentally.  The  region  $R'$ is then deleted by an $H_2^-$ move (subcase b-i, 
on the left of the figure) or by 
 an $H_1^-$ move (subcase b-ii, on the right of the figure).
In case b-i) both moves can be dispensed with since $S_1$ is homeomorphic  to $S_3$. 
In case b-ii)  the result,  $S_3$, could have been created with an $H_1^+$ move. So $S_1$ is joined to $S_3$ by an edge.

In case c) the regions $R$ and $R'$ have two points in common. 
In case c-i),  on the left of the figure, let $S=(\Sigma, D)$ be a doodle diagram which is obtained from 
$S_2$ by removing the circle bounding $R \cup R'$. Then   
$S_1$ (or  $S_3$, resp.)  is obtained from $S$ by adding a floating component $C_1$ (or $C_3$, resp.). 
Let $S'_2$ be the disjoint union of $S$ and a trivial doodle diagram with one component.   
Then, as in the case 1) of the proof of Lemma~\ref{lem:floating}, we see that  $S_1\searrow S'_2\swarrow S_3$. 
In case c-ii),  on the right of the figure, let $S_i = (\Sigma, D_i)$ $(i=1,2,3)$ be the doodle diagrams, and  
let $C_2$ and $C'_2$ be the circles of $D_2$ facing $R$ and $R'$ in the figure, respectively.  In $D_1$ their corresponding circles $C_1$ and $C'_1$ are lying such that $C_1$ is inside of $C'_1$, and in $D_3$ their corresponding circles $C_3$ and $C'_3$ are lying such that $C'_3$ is inside of $C_3$.  
Note that $D_i = (D_2 \setminus (C_2 \cup C'_2) ) \cup (C_i \cup C'_i)$ for $i=1,3$.  
Let $\Sigma_0$ be the component of the surface $\Sigma$ which contains $C_2$ and $C'_2$.   
There are two subcases: 
\begin{itemize}
\item c-ii-1) $(D_2 \setminus (C_2 \cup C'_2) ) \cap \Sigma_0 \neq \emptyset$. 
\item c-ii-2) $(D_2 \setminus (C_2 \cup C'_2) ) \cap \Sigma_0 = \emptyset$. 
\end{itemize}

In case c-ii-1), 
let $S'_2$ be the disjoint union of $(\Sigma, D_2 \setminus (C_2 \cup C'_2)) $ and a trivial doodle diagram with two components. We see that  
there is a path $S_1 \searrow S'_1 \searrow  S'_2 \swarrow S'_3 \swarrow S_3$ for some $S'_1$ and 
$S'_3$ as in the case 1) of the proof of Lemma~\ref{lem:floating}. 

In case c-ii-2), 
let $S'_2$ be the disjoint union of $(\Sigma \setminus \Sigma_0, D_2 \setminus (C_2 \cup C'_2)) $ and a trivial doodle  diagram with two components. By the argument of the proof of Lemma~\ref{lem:floating}, we see that there is a path 
$S_1 \searrow S'_1  \searrow  S'_2 \swarrow S'_3 \swarrow S_3$ when $\Sigma_0$ is a $2$-sphere, and there is a path 
$S_1 \searrow S'_1  \searrow  \cdots \searrow 
S'_2 \swarrow \cdots \swarrow   S'_3 \swarrow S_3$ when $\Sigma_0$ has positive genus.  

We now consider mixtures of  $h^{\pm 1}$ and $H_1^{\pm 1}, H_2^{\pm 1}$ moves. 

Suppose we have a simple peak $S_1\swarrow S_2\searrow S_3$. There are various cases to consider. Let $S_1\swarrow S_2$ be an $H_1^+$ move and $S_2\searrow S_3$ an $h^-$ move. The expansion creates a monogon and the handle elimination takes place along an essential curve which is disjoint from the diagram and hence disjoint from the monogon. It follows that the two operations can be reversed.

A similar argument holds if the initial expansion creates a bigon with an $H_2^+$ move.

Now consider $S_1\swarrow S_2\searrow S_3$ in which the first expansion makes a handle from the  boundary of disks $D_1$ and $D_2$ and the second move eliminates a bigon (or monogon). In either case this must be disjoint from the two disks in order to happen so the two moves can be reversed.

Finally, suppose the expansion, $S_1\swarrow S_2$, creates a handle and the collapse, $S_2\searrow S_3$, deletes a handle. This means that in the middle diagram $S_2$ there is a simple closed curve $b_1$ which is the belt of the added handle and a simple closed curve $b_2$ upon which the surgery takes place. Both curves are disjoint from the  doodle diagram $D_2$ on $\Sigma_2$ with $S_2 = (\Sigma_2, D_2)$. 

If $b_1$ and $b_2$ are disjoint then we can reverse the operations. If $b_1$ and $b_2$ are not disjoint then we can assume that they meet transversely. Let $A_1$ and $A_2$ be annular neighbourhoods of $b_1$ and $b_2$. If these are sufficiently thin then,

\begin{enumerate}
\item  the components of $A_1\cap A_2$ are square neighbourhoods of the intersection points of  $b_1$ and $b_2$,
\item $A_1\cup A_2$ is a connected surface with boundary,
\item $A_1\cup A_2$ is disjoint from the doodle diagram $D_2$. 
\end{enumerate}

Let the components of $\partial (A_1\cup A_2)$ be $c_1\cup c_2\cup\cdots\cup c_n$.

If each $c_i$ is inessential then it spans a disc $U_i$ in the surface $\Sigma_2$. So the (topological) components of the doodle diagram $D_2$ each lie in a disc. (It follows that the doodle is planar.) 
In this case, there is a sequence of $h^-$ moves 
transforming  $S_1$ into a doodle diagram $S'=(\Sigma', D')$ which is a disjoint union of doodle diagrams on $2$-spheres, and there is a sequence of $h^-$ moves  transforming $S_3$ into the same $S'=(\Sigma', D')$.

Otherwise, at least one of the boundary components is essential: call it $c$. Then $c$ is disjoint from $b_1$ and $b_2$. 

Case 1) Suppose that $c$ is a non-separating loop of $\Sigma_2$  or that $c$ is a separating loop and 
each component of the surface obtained by the surgery on $c$ intersects with the diagram $D_2$.  
Let $S'$ be the result of surgery on $c$ in $S_2$, which is an admissible doodle diagram.  Let $S''$, ($S'''$) be the result of surgery on $b_1$, ($b_2$) in $S'$. Note that $S''$, ($S'''$) is also the result of surgery on $c$ in $S_1$, ($S_3$).

So there is a path $S_1\searrow S''\swarrow S'\searrow S'''\swarrow S_3$ and an edge $S_2\searrow S'$ which implies the local diamond condition.

It follows that the conditions of the proof reduction graph are satisfied and the theorem is proved. 

Case 2) Suppose that $c$ is a separating loop and one of the component of the surface obtained by the surgery on $c$ does not intersect with $D_2$.  In this case, the component missing $D_2$ is a closed connected surface with positive genus. This implies that we can first apply $h^{-1}$ moves to $S_1$ and $S_3$ along simple closed curves on the surface killing the genus, and we can reduce the case into the case that $c$ is inessential.  
\end{proof}

\begin{definition}{\rm 
A doodle diagram is {\bf minimal} if the interior of every region is 
simply connected and there are no monogons and bigons. 
}\end{definition}

Theorem~\ref{thm:proofreductiongraph} implies a Kuperberg type theorem, see \cite{GK}.

\begin{theorem}\label{thm:minimal}
If two minimal diagrams represent the same oriented (or unoriented) doodle then they are homeomorphic.
\end{theorem}

\begin{proof} 
Since there are no monogons and bigons, no $H_1^-, H_2^-$ moves can be initiated. Because the regions are simply connected the same is true for $h^-$ moves. It follows that a minimal diagram represents a root and roots are unique.
\end{proof}

A doodle is called {\bf trivial} if it has a trivial doodle diagram as a representative.  

By Theorem~\ref{thm:minimal}, a doodle is trivial if and only if its unique minimal diagram is a trivial doodle diagram.  

The poppy  and  the Borromean doodle are non-trivial because the diagrams depicted in Figure~\ref{fig01Poppyfig02Borromean} are minimal diagrams which are not trivial doodle diagrams.  
Similarly the Hopf doodle is non-trivial because the diagram on the torus 
with one square region is a minimal diagram which is not a trivial doodle diagram.

Theorem~\ref{thm:proofreductiongraph} also implies that a minimal diagram is characterized as a diagram with minimum crossing number and the ambient surface has minimum genus and  maximum  component number as described below.  

For a diagram $(\Sigma, D)$ let $c(D)$, $g(\Sigma)$ and ${\rm comp}(\Sigma)$ stand for the number of crossings of $D$, the genus of $\Sigma$ and the number of connected components of $\Sigma$.    

\begin{theorem}\label{thm:characterize}
Let $(\Sigma, D)$ be a doodle diagram. The following conditions are equivalent. 
\begin{itemize} 
\item[$(1)$] $(\Sigma, D)$ is a minimal diagram. 
\item[$(2)$] For any diagram $(\Sigma', D')$ representing the same doodle as $(\Sigma, D)$, 
three inequalities 
$c(D) \leq c(D')$, $g(\Sigma) \leq g(\Sigma')$ and $
{\rm comp}(\Sigma) \geq {\rm comp}(\Sigma')$ 
hold. 
\end{itemize} 
\end{theorem}  

\begin{proof}
Note that a doodle diagram is minimal if and only if it is a root of the graph $\cal D$ of doodle diagrams. 
(1) $\Rightarrow$ (2):  Suppose (1). 
By Theorem~\ref{thm:proofreductiongraph}, if $(\Sigma, D)$ and $(\Sigma', D')$ are not homeomorphic then there exists a 
descending path from  $(\Sigma', D')$ to $(\Sigma, D)$.  This implies the inequalities. 
(2) $\Rightarrow$ (1):  It is obvious.  
\end{proof}

The {\bf genus} of a doodle is the minimum genus among all surfaces on which the doodle has a representative.   
A doodle with genus $0$ is a planar doodle.   
By Theorem~\ref{thm:characterize} we have the following. 

\begin{corollary}\label{cor:genus}
The genus of a doodle is the genus of the surface of a (unique) minimal diagram of the doodle. 
\end{corollary}

\section{Examples of Minimal Doodle Diagrams} \label{sect:example}

In this section we will look at minimal diagrams of doodles. Here we consider unoriented doodles. 

\subsection{Cell Decompositions of Surfaces from Minimal Doodle Diagrams}

Let $D$ be a non-trivial minimal doodle diagram on a connected surface of genus $g$. Then the crossings, edges and regions form a cell decomposition of the ambient surface. Let $V$ be the number of crossings, $E$ the number of edges and $F$ the number of regions. Let $F_i$ be the number of $i$-gon regions, $i=3,4,\ldots$.
\begin{lemma}With the notation above,
\begin{itemize}
\item [$(I_1)$] $E=2V$
\item [$(I_2)$] $F=V+2-2g$
\item [$(I_3)$] $V-6+6g=F_4+2F_5+3F_6+\cdots$
\end{itemize}
\end{lemma}

\begin{proof} 
Every vertex contributes $4$ to the number of edges twice over. This proves identity $I_1$.
The formula for the Euler-Poincar\'e number is $V-E+F=2-2g$.  Substituting $I_1$ gives $I_2$.
Each $i$-gon contributes $i$ edges to $E$. Each edge is the edge of $2$ regions. So
$4V=2E=3F_3+4F_4+\cdots$ But $6+3V-6g=3F=3F_3+3F_4+\cdots$. Taking the difference implies $I_3$.
\end{proof}

\begin{theorem}
Consider non-trivial minimal doodle diagrams on the $2$-sphere.
\begin{itemize}
\item[$(1)$] They  have at least 6 crossings.
\item[$(2)$] There is only one with 6 crossings: the Borromean doodle.
\item[$(3)$] There are none with 7 crossings.
\item[$(4)$] There is only one with 8 crossings: the poppy.
\end{itemize}
\end{theorem}

\begin{proof} 
(1) It follows immediately from identity $I_3$.  

(2) If $V=6$ then $I_3$ implies that there are only triangular regions. These form the octahedral decomposition of the 2-sphere. 

(3) 
If $V=7$ then $I_3$ implies that there exists a single tetragonal region and other regions are all triangular regions.  
Let $\{v_0, v_1, v_2, v_3\}$ be the vertices (crossings) of the tetragonal region, and 
$\{u_0, u_1, u_2\}$ the other three vertices.  
Note that two successive edges among the four edges $v_i v_{i+1}$ $(i=0, 1, 2, 3)$ bounding the tetragonal region never bound a triangular region.  
Thus for each edge $v_i v_{i+1}$ $(i=0, 1, 2, 3)$ of tetragonal region, there is a 
vertex $w_i$ $(i=0, 1, 2, 3)$ $\in \{u_0, u_1, u_2\}$ such that $v_i v_{i+1} w_i$ is a 
triangular region.  
Without loss of generality, we may assume that $w_0 = w_2=x$, 
$w_1=a$ and $w_3=b$ as in Figure~\ref{fig22nominimal}, where the  tetragonal region is outer-most. 
Two edges which connects to the crossing $a$ (or $b$) are not depicted in this figure.
It is impossible to draw such edges under this circumstance. 
 
%\vskip .5cm 

%\diagram{No minimal planar doodle diagram with 7 crossings}
    \begin{figure}[h]
    \centerline{\epsfig{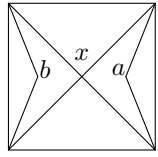}}
    \vspace*{8pt}
    \caption{No Minimal Planar Diagram with 7 Crossings}\label{fig22nominimal}
    \end{figure}

(4) $I_3$ implies that either 
(4-1) $F_4=2$ and $F_i=0$ for $i>4$ or 
(4-2) $F_5=1$ and $F_i=0$ for $i=4$ and for $i >5$. 
By a similar argument with (3) we see that case (4-1) does not occur.   
In case (4-1), there are two tetragonal regions.  If they are disjoint, then we have the poppy.  If they share one or two vertices, we have a contradiction.
\end{proof}

 \subsection{Infinite Sequences of Planar Doodles}
 Because we can recognize non-trivial doodles without $1$-gons or $2$-gons it is easy to invent sequences of different doodles. Here we define some sequences of planar doodles by giving sequences of minimal planar diagrams 
 whose early members have other interpretations.

\begin{example}{\rm 
There is a sequence of doodles $B_3, B_4,\ldots$ starting with the Borromean doodle, $B_3$ and the poppy, $B_4$.   Let the vertices of the two concentric $n$-gons be $X_1X_2\ldots X_n$ and $Y_1Y_2\ldots Y_n$. Construct the squares $X_iX_{i+1}Y_{i+1}Y_i$, $i=1,\ldots, n$ cyclically mod $n$. Join the diagonals $X_i$ to $Y_{i+1}$ sequentially to create the triangles.
This defines $B_n$. It has $2n$ vertices, $2n$ triangular faces and 2 $n$-gon faces. The number of components is $3$ if $n$ is divisible by $3$ and $1$ otherwise. We call these doodles the {\bf generalized Borromean} doodles.
}\end{example} 

\begin{example}{\rm 
Another two sequences can be defined by taking $2$ concentric $n$-gons separated by a concentric $2n$-gon and filling in the annular regions with alternate squares and triangles. This can be done in two ways so that at the $2n$-gon each square or triangle in one annulus abutts a single square or triangle in the other annulus. So taking nomenclature from the classification of polyhedra, we have {\bf Gyro}, $C'_n$, which has the squares abutted by a common edge to the triangles whilst {\bf Ortho}, $C''_n$, has the squares abutted by a common edge to the squares and the triangles abutted to the triangles. Both $C'_n$ and $C''_n$ have $4n$ vertices, $F_n=2$ and $F_3=F_4=2n$. The doodles  $C'_n$ and $C''_n$ for $n>3$ can be distinguished by the number of components. For $n=3$ both $C'_3$ and $C''_3$ have $4$ components and a similar number of triangular and square regions. They are illustrated in Figure~\ref{fig23CC} and can be distinguished by the combinatorics of their regions.

%\diagram{$C'_3$ and $C''_3$}
    \begin{figure}[h]
    \centerline{\epsfig{file=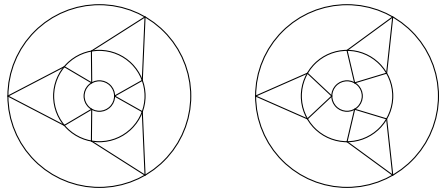, height=2.5cm}}
    \vspace*{8pt}
    \caption{$C'_3$ and $C''_3$}\label{fig23CC}
    \end{figure}

}\end{example} 

We can describe $C'_3$ and $C''_3$ as follows. Consider the Borromean doodle, $B_3$, with circular components (see Figure~\ref{fig01Poppyfig02Borromean}). Now draw a circle separating the innermost triangular region from the outermost  triangular region. This describes $C'_3$.  To obtain $C''_3$, perform an $R_3$ move on the innermost triangular region. 
See Figure~\ref{fig24CCcircular}. 

%\diagram{$C'_3$ and $C''_3$ with circular components}
    \begin{figure}[h]
    \centerline{\epsfig{file=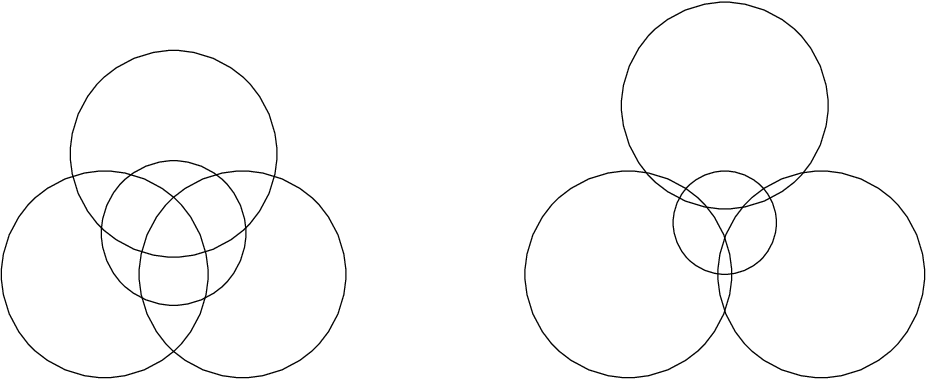, height=2.5cm}}
    \vspace*{8pt}
    \caption{$C'_3$ and $C''_3$ with Circular Components}\label{fig24CCcircular}
    \end{figure}

If we remove any component of $C'_3$ we get the Borromean doodle. On the
other hand, if we remove the innermost circle from $C''_3$ we get the trivial doodle. This is another proof that  $C'_3$ and $C''_3$ are distinct.

%They are examples of a sequence of doodles with  circular components rotated about an origin $O$ such that consecutive circles meet in lunes and there is  a central triangular region containing $O$.

%The two doodles $C'_3$ and $C''_3$ both have 4 components, V=12 and with 8 triangular faces and 6 square faces. However they are different because one of them, $C'_3$,  has the property that if you remove any circle you have the Borromean rings and $C''_3$ doesn't have this property.
%\diagram{Stripping $C'_3$ of one component}
% In the next figure the doodle on the left is trivial. 
%\diagram{Stripping $C''_3$ of one component}

\begin{lemma} 
\begin{itemize}
\item[$(1)$] The doodle $C'_n$ has four components if $n$ is divisible by 3. Otherwise $C'_n$ has two components.
\item[$(2)$] The doodle $C''_n$ has $n+1$ components.
\end{itemize}
\end{lemma}

\begin{proof} 
Firstly note that the central 2$n$-gon is one of the components in both cases. 
(1)  Let the vertices around one of the $n$-gons be  $P_1\ldots P_{n}$. A component of $C'_n$ containing the edge $P_iP_{i+1}$, contains the edge $P_{i+3}P_{i+4}$. 
(2)  For $C''_n$  each component which isn't the central 2$n$-gon is a hexagon containing the edge $P_iP_{i+1}$, $i=1,\ldots n$ mod $n$, and there are $n$ of these.
\end{proof}

%\lemma{\nl
%1. The doodle $C'_n$ has four components if $n$ is divisible by 3. Otherwise $C'_n$ has two components.\nl
%2. The doodle $C''_n$ has five components if $n$ is divisible by 5. Otherwise $C''_n$ has one component.\nl}
%{\bf Proof:} \nl 
%Note that the Gyro family, $C'_n$, has squares constructed on the faces of the inner n-gon, whilst the Ortho family, $C''_n$, has triangles constructed on the inner n-gon.  This observation shows that, for $C'_n$, the central 2n-gon forms a component by itself.  Moreover, as demonstrated in the following diagrams, a component of $C'_n$ containing a face of the outer n-gon contains every third face (modulo n) of the outer n-gon and a component of $C''_n$ containing a face of the outer n-gon contains every fifth face (modulo n) of the outer n-gon. \qed
%\diagram{$C'_7$ and $C''_7$}

%?????????????? anyone have a sequence of doodles with genus 1 ???????????????

\begin{remark}{\rm 
{\bf (A Note on Planar Doodles and Polyhdra)}
There is a bijection between minimal planar doodles and the 1-skeleta of 3-dimensional polehdra whose vertices have valency four. It is well known that the Borromean doodle, $B_3$  is the 1-skeleton of the octahedron. In general $B_n$ is  the 1-skeleton of the $n$-gon antiprism,
see \cite{C} for definitions.

Furthermore $C'_3$ is the 1-skeleton of the cuboctahedron and $C''_3$ is the 1-skeleton of the anticuboctahedron or triangular orthobicupola which is Johnson's solid $J_{27}$. 
$C'_4$/$C'_5$ are also 1-skeleta of Johnson solids: $J_{29}$/$J_{31}$.  $C''_4$/ $C''_5$  are the 1-skeleta of the square gyrobicupola/pentagonal gyrobicupola  and $J_{28}$/$J_{30}$ are the 1-skeleta of the square orthobicupola/pentagonal  orthobicupola. 
Further bicupola for $n>5$ can not be realised with regular faces.

We are grateful to Peter Cromwell for this information. 
}\end{remark}

\subsection{The $\pm 1$ Construction}
Let $D$ be a doodle diagram on a surface $\Sigma$. Suppose that $D$ has a region $R$ with at least four edges and chose two disjoint edges $e_1$ and $e_2$ of $R$. Remove the interiors of the edges and join the dangling vertices with two diagonal arcs meeting at a new point $X$ in the interior of $R$. This creates a new doodle diagram $D'$. We write $D{+1\atop\to}D'$ or $D'{-1\atop\to}D$ and call $D$ an {\bf ancestor} of $D'$ and $D'$ a {\bf descendant} of $D$. The number of components of $D'$ changes from that of 
 $D$ by $0$ or $\pm 1$, 
 depending how  $e_1$ and $e_2$  are oriented and placed.

\begin{lemma}
Let $D'$ have an ancestor $D$.   Then $D'$ is minimal if and only if $D$ is minimal.
\end{lemma}

Figure~\ref{fig25descendantsPoppy}  illustrates how the poppy is the ancestor of (the unique) minimal diagram with 9 crossings and (the unique) two component minimal diagram with 10 crossings.

%\diagram{Descendants of the Poppy}
    \begin{figure}[h]
    \centerline{\epsfig{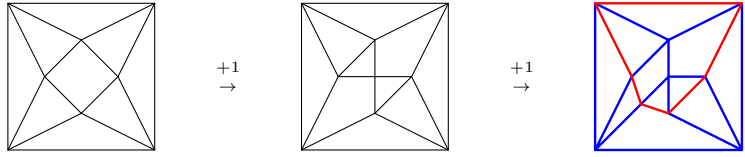}}
    \vspace*{8pt}
    \caption{Descendants of the Poppy}\label{fig25descendantsPoppy}
    \end{figure}

\begin{definition}{\rm 
A diagram is {\bf fundamental} if it is a minimal diagram without ancestors.  
}\end{definition}

The following lemma gives a procedure for recognising whether a diagram is fundamental or not. 

\begin{lemma}
If a minimal diagram has an ancestor, then there is a crossing $X$ satisfying one of the following.
\begin{itemize}
\item[$(1)$] There are two distinct regions appearing in a diagonal position about $X$ such that they are a $p'$-gon and a $q'$-gon for some $p' >3$ and $q' >3$. 
\item[$(2)$] There is a region appearing in a diagonal position about $X$ such that it is an $r'$-gon for some $r' >5$. 
\end{itemize}
\end{lemma}

\begin{proof} 
Let $D'$ be a descendant of $D$.  Let $e_1$ , $e_2$ and $R$ be the edges and the region of $D$ that were used to produce $D'$.  (1) Consider a case that the regions  containing $e_1$ and $e_2$, beside $R$, are distinct.  Suppose they are $p$-gon and $q$-gon. Since $D$ is minimal, $p>2$ and $q>2$. After applying the $+1$ construction, these regions become a $p'$-region and a $q'$-region with $p'=p+1$ and $q'= q+1$.  (2) Consider a case that the regions containing $e_1$ and $e_2$, beside $R$, are the same region of $D$.  Suppose it is an $r$-gon. Since 
 the boundary of this region contain $e_1$ and $e_2$, we have $r >3$.  By the $+1$ construction, this region becomes an $r'$-region with $r'=r+2$.
 \end{proof}

\begin{theorem}The generalised Borromean doodles are all fundamental. 
\end{theorem}

\begin{proof}  Their two regions with $>3$ edges are protected by a ring of triangles. 
Thus there is no crossing satisfying the condition of the previous lemma. 
\end{proof}

\section{Virtual Doodles} \label{sect:virtual} 

In this section we define a virtual doodle, which is an equivalence class of a virtual doodle diagram on the plane.

An {\bf oriented} (or {\bf unoriented}) {\bf virtual doodle diagram} (on the plane) is a generically immersed oriented (or unoriented) circles on the plane such that 
some of the crossings are decorated by small circles.  It is often regarded as an oriented 4-valent graph on the plane.  When it is oriented, each crossing is as illustrated in Figure~\ref{fig06aflatandvirtual}.

%\diagram{A Flat and a Virtual Crossing}
    \begin{figure}[h]
    \centerline{\epsfig{file=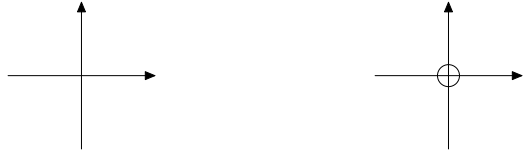, height=2.0cm}}
    \vspace*{8pt}
    \caption{A (Flat) Crossing and a Virtual Crossing}\label{fig06aflatandvirtual}
    \end{figure}

A crossing which is not encircled is called a {\bf flat} crossing, a {\bf real} crossing 
or sometimes simply a crossing. 
An encircled crossing is called a {\bf virtual} crossing. 

We consider moves on such diagrams.

We firstly consider moves $R_1$ and $R_2$ illustrated in Figure~\ref{fig07FR_1fig08FR_2}, where we should consider all possible orientations in the oriented case. They involve flat crossings, which are flat versions of the first Reidemeister move and the second. 
The shaded areas, a monogon and a bigon, have interiors disjoint from the rest of the diagram. We can divide these into two types: $R_1^+$ which creates the monogon and $R_1^-$ which deletes it. Similarly we have $R_2^+$ and  $R_2^-$.

%\diagram{The allowed move $FR_1$}
%\diagram{The allowed move $FR_2$}
    \begin{figure}[h]
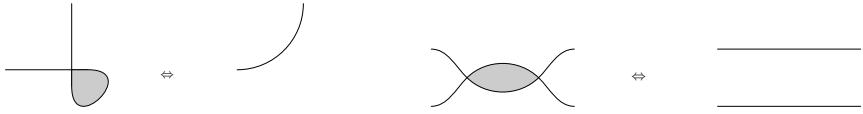

    \centerline{\epsfig{file=doodlefig07FR1.eps, height=1.4cm}
    \qquad \qquad 
    \epsfig{file=doodlefig08FR2.eps, height=0.8cm}
    }
    \vspace*{8pt}
    \caption{Moves $R_1$ (Left) and $R_2$ (Right)}\label{fig07FR_1fig08FR_2}
    \end{figure}

Moves $VR_1$, $VR_2$ and $VR_3$ are illustrated in Figures~\ref{fig10VR_1fig11VR_2} and \ref{fig12VR_3fig13VR_4}(Left), which involve virtual crossings. As in the flat case, there are versions $VR_1^{\pm1}$ and $VR_2^{\pm1}$. For $VR_3$ in the oriented case, we can assume that all three arcs are oriented from left to right, (braid like), see \cite{F}.

%\diagram{$VR_1$}
%\diagram{$VR_2$}
    \begin{figure}[h]
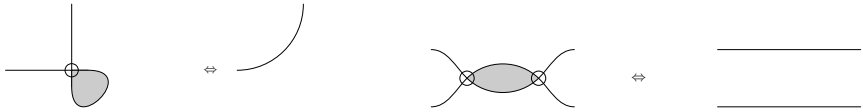

    \centerline{\epsfig{file=doodlefig10VR1.eps, height=1.4cm}
    \qquad \qquad 
    \epsfig{file=doodlefig11VR2.eps, height=0.8cm}
    }
    \vspace*{8pt}
    \caption{Moves $VR_1$ (Left) and  $VR_2$ (Right)}\label{fig10VR_1fig11VR_2}
    \end{figure}

%\diagram{$VR_3$}
%\diagram{The allowed move $VFR_3$}
    \begin{figure}[h]
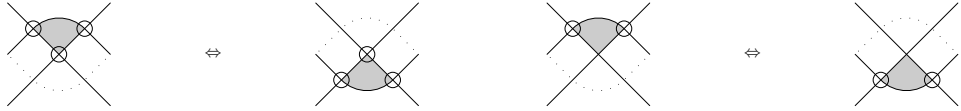

    \centerline{\epsfig{file=doodlefig12VR3.eps, height=1.4cm}
    \qquad \qquad 
    \epsfig{file=doodlefig13VFR3.eps, height=1.4cm}
    }
    \vspace*{8pt}
    \caption{Moves $VR_3$ (Left) and  $VR_4$ (Right)}\label{fig12VR_3fig13VR_4}
    \end{figure}

Finally, we have a mixed move $VR_4$ illustrated in Figure~\ref{fig12VR_3fig13VR_4}(Right) 
in which two consecutive virtual crossings appear to move past a flat crossing.

\begin{definition}{\rm 
Two oriented (or unoriented) virtual doodle diagrams are {\bf equivalent} if they are related by a sequence of the moves   
$R_1$, $R_2$, $VR_1$, $VR_2$, $VR_3$ and $VR_4$, modulo isotopies of the plane.   
An {\bf oriented} (or {\bf unoriented}) {\bf virtual doodle} is an equivalence class of an oriented (or unoriented) virtual doodle diagram.  
}\end{definition}

The move $R_3$ depicted in Figure~\ref{fig09FR_3fig14FVR_3}(Left), which is the flat version of the third Reidemeister move,  is forbidden in doodle theory. The move $FVR_3$ depicted in Figure~\ref{fig09FR_3fig14FVR_3}(Right), in which two consecutive flat crossings appear to move past a virtual crossing is also forbidden.

%\diagram{The forbidden move $FR_3$}
%\diagram{The forbidden move $FVR_3$}
    \begin{figure}[h]
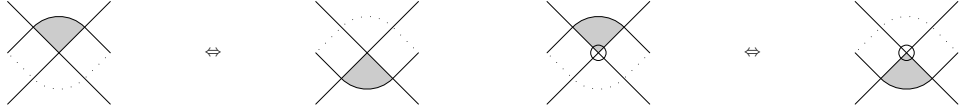

    \centerline{\epsfig{file=doodlefig09FR3.eps, height=1.4cm}
    \qquad \qquad 
    \epsfig{file=doodlefig14FVR3.eps, height=1.4cm}
    }
    \vspace*{8pt}
    \caption{Forbidden moves $R_3$ (Left) and $FVR_3$ (Right)}\label{fig09FR_3fig14FVR_3}
    \end{figure}

\begin{remark}{\rm 
If we allow the forbidden move $R_3$ then the theory of doodles collapse to the theory of flat virtual knots and links.    The flat Kishino knot diagram is depicted in Figure~\ref{fig15Kishino}.  It represents a non-trivial flat virtual knot, called the flat Kishino knot.  (The original Kishino knot diagram is a diagram of a virtual knot.  
For a proof of the non-triviality as a virtual knot, see \cite{BF}, \cite{KS}.)  The flat version was proved to be non-trivial in \cite{Kad} and also in \cite{FTu}.  Thus the flat Kishino knot is also non-trivial as a flat doodle.  Figure~\ref{fig16Kishinosurface} illustrates the Kishino knot as a doodle representative on a genus-$2$ surface. 
 
If we continue to forbid $R_3$ but allow the previously forbidden move  $FVR_3$ then we get the theory of welded doodles.   So far we have not been able to prove that non-trivial examples exist.
}\end{remark} 

%%%%%%
%%%%%%
%$$\epsfbox{doodlefigk2.eps}$$
%%$$ Figure ``doodlefigk2'' here $$ 
%\centerline{\includegraphics[width=7cm]{doodlefigk2}}
%%%%%%
%%%%%%

%\diagram{The Flat Kishino as a Flat Virtual Diagram}
    \begin{figure}[h]
    \centerline{\epsfig{file=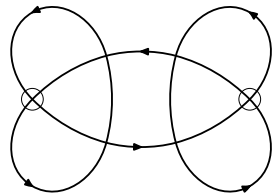, height=2.5cm}}
    \vspace*{8pt}
    \caption{The Flat Kishino as a Flat Virtual Diagram}\label{fig15Kishino}
    \end{figure}

%\diagram{The Flat Kishino  on a Surface of Genus 2}
    \begin{figure}[h]
    \centerline{\epsfig{file=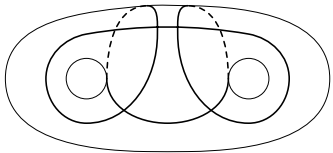, height=2.0cm}}
    \vspace*{8pt}
    \caption{The Flat Kishino  on a Surface of Genus 2}\label{fig16Kishinosurface}
    \end{figure}

Call a sub-path of the diagram which only passes only through virtual crossings a {\bf virtual path}. It is a consequence of  $VR_1$, $VR_2$, $VR_3$ and $VR_4$ that a virtual path can be moved any where in the diagram keeping its end points fixed provided that the new crossings engendered by this are virtual. Kauffman calls this the {\bf detour} move, \cite{K}.

\section{Virtual Doodles and Doodles Are the Same} \label{sect:virtualsame}

We show that there is a bijection between oriented (or unoriented) virtual doodles (on the plane) and oriented (or unoriented) doodles (on surfaces). 
The idea is similar to that in \cite{KK} (and \cite{CKS}) showing a bijection between virtual links and 
abstract links (or stable equivalence classes of link diagrams on surfaces).  

In this section, we denote a virtual doodle diagram (on $\R^2$) by $K$ and a doodle diagram (on a surface) by $D$.  

Let $K$ be a virtual doodle diagram with $m$ flat crossings.  Let $N_1, \dots, N_m$ be regular neighbourhoods of the flat crossings and put $W:= {\rm Cl}(\R^2 \setminus \cup_{i=1}^m N_i)$.  The intersection  $K \cap W$ is a union of arcs and loops immersed in $W$.  (Each intersection point of $K \cap W$  is a virtual crossing of $K$. A loop appears when $K$ has a component on which there are no flat crossings.) 
Let $K'$ be another virtual doodle diagram with $m$ flat crossings.  
Let $\sigma$ be a bijection from the set of 
crossings of $K$ to that of $K'$.  Using an isotopy of $\R^2$, we assume that each crossing of $K$ and the corresponding crossing of $K'$ under $\sigma$ are the same point of $\R^2$ and that 
$K$ and $K'$ are identical in $N_1, \dots, N_m$.  We say that $K$ and $K'$ have the same {\bf  Gauss data} with respect to $\sigma$ if there is a bijection $\tau$ from  the set of arcs and loops of $K \cap W$ to that of $K' \cap W$ such that for every arc $e$ of $K \cap W$ the endpoints of $e$ equal the endpoints of $\tau(e)$.  

\begin{lemma}\label{lem:Gauss}
Let $K$ and $K'$ be  virtual doodle diagrams with $m$ flat crossings. The following conditions are equivalent: 
\begin{itemize} 
\item[$(1)$]  
They have the same Gauss data with respect to a bijection between their flat crossings. 
\item[$(2)$] 
They  are related by a finite sequence of detour moves modulo isotopies of $\R^2$. 
\item[$(3)$]  
They  are related by a finite sequence of 
moves $VR_1$, $VR_2$, $VR_3$ and $VR_4$ modulo isotopies of $\R^2$.  
\end{itemize}
\end{lemma}

\begin{proof} 
The equivalence between (1) and (2) is obvious.  The equivalence between (2) and (3) is due to Kauffman \cite{K} (cf. \cite{KK}). 
\end{proof}

Let $K$ be a  virtual doodle diagram with $m$ flat crossings. We continue with the definition of $N_1, \dots, N_m$ and $W$ as before.

Thickening the arcs and loops in $W$, we obtain  bands and annuli immersed in $W$ whose cores are $K \cap W$.  The union of these bands and annuli with $N_1, \dots, N_m$ is a compact oriented surface immersed in $\R^2 = \R^2 \times \{0\} \subset \R^3$.  Here we assume the orientation of the surface is induced from the orientation of $\R^2$.  
Replacing it in neighbourhoods of virtual crossings as in Figure~\ref{fig17interpretation} in $\R^3$, we obtain a compact oriented surface, $F$,  embedded in $\R^3$  and a diagram, $D_F$, with flat crossings on it.  
We denote the pair $(F, D_F)$ by $\varphi(K)$.  Here the pair $(F, D_F)$ is considered up to orientation-preserving homeomorphic equivalence, and we ignore how it is embedded in $\R^3$.

%%%%%%
%%%%%%
%$$\epsfbox{doodlefigk1.eps}$$
%$$ Figure ``doodlefigk1'' here $$ 
%\centerline{\includegraphics[width=5cm]{doodlefigk1}}
%%%%%%
%%%%%%

%\diagram{Topological Interpretation of a Virtual Crossing}
    \begin{figure}[h]
    \centerline{\epsfig{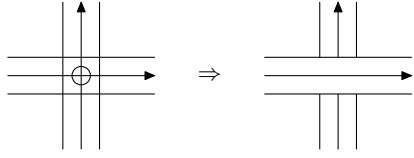}}
    \vspace*{8pt}
    \caption{Topological Interpretation of a Virtual Crossing}\label{fig17interpretation}
    \end{figure}

\begin{lemma}
If $K'$ is obtained from $K$ by $VR_1$, $VR_2$, $VR_3$ or $VR_4$, then $\varphi(K)$ is homeomorphic to $\varphi(K')$.
\end{lemma}

\begin{proof} 
 By a move $VR_1$, $VR_2$, $VR_3$ or $VR_4$,  $N_1, \dots, N_m$ are preserved, and there is a bijection from the set of arcs and loops of $K \cap W$ to that of $K' \cap W$.  
It induces a homeomorphism from the union of the  bands and annuli  in $\varphi(K)$ and that in $\varphi(K')$.  Sending $N_1, \dots, N_m$ identically, we have a homeomorphism from $\varphi(K)$ to $\varphi(K')$.  
\end{proof}

\begin{lemma}
If $K'$ is obtained from $K$ by $R_1^{\pm}$ or  $R_2^{\pm}$, then there is a closed connected oriented surface $\Sigma$ and doodle diagrams $D$ and $D'$ on $\Sigma$ such that 
$(N(D, \Sigma), D)$ is homeomorphic to $\varphi(K)$, $(N(D', \Sigma), D')$ is homeomorphic to $\varphi(K')$ and $D'$ is obtained from $D$ by 
$H_1^{\pm}$ or $H_2^{\pm}$ on $\Sigma$. 
\end{lemma}

\begin{proof} 
It suffices to consider the case where $K'$ is obtained from $K$ by $R_1^{-}$ or $R_2^{-}$.  Let $\varphi(K) =(F, D_F)$.  Attaching $2$-disks to $F$ along the boundary, we have a closed oriented surface $\tilde F$ in which there is a monogon or a bigon where $H_1^{-}$ or $H_2^{-}$ can be applied.    
\end{proof}

So if  $K$ is a  virtual doodle diagram and $\varphi(K) = (F, D_F)$ is as above then by  
attaching $2$-disks, annuli, or even any compact connected oriented surfaces to $F$ along the boundary, we have a closed oriented surface $\Sigma$ and an (admissible) doodle diagram $D$ such that 
$(N(D, \Sigma), D)$ is homeomorphic to $\varphi(K)$.  We call it a {\bf doodle diagram associated to $K$}.  
It is not unique as a representative, however it is unique as a doodle.  We call it, as a doodle, the {\bf  doodle associated to $K$}.  

\begin{theorem} \label{thm:bijection}
We have a bijection $\Phi $ from the family of oriented (or unoriented) virtual doodles to the family of oriented (or unoriented) doodles by 
defining $\Phi ([K])$ to be the doodle associated to $K$.    
\end{theorem}

\begin{proof} We consider the oriented case.  The unoriented case follows from the oriented case.   

The well-definedness of $\Phi $ follows from the previous lemmas.  We shall prove that $\Phi $ is surjective and then injective.   

Let $\pi : \R^3 \to \R^2$ be the projection $(x,y,z) \mapsto (x,y)$. 

Let $(\Sigma, D)$ be a doodle diagram $D$ on a surface $\Sigma$. 
By an abuse of notation let $N_1, \dots, N_m$ also stand for the regular neighbourhoods of crossings of $D$ in $\Sigma$, 
and put $X := {\rm Cl}(\Sigma \setminus \cup_{i=1}^m N_i)$.  The intersection $D \cap X$ is a union of arcs and loops embedded in $X$.  The neighborhoods of the arcs and loops in $X$ are bands and annuli in $X$, and the union of these bands and annuli together with $N_1, \dots, N_m$ forms a regular neighborhood $N(D)$ of $D$ in $\Sigma$.  Consider an embedding $\tilde g : N(D) \to \R^3$ such that $\tilde g (\cup_{i=1}^m N_i) 
\subset \R^2 \times \{0\}$ and $g := \pi \circ \tilde g : N(D) \to \R^2$ is an orientation-preserving immersion whose singularity set is a union of transverse intersections among bands and annuli. 
Regard $g(D)$ as a  virtual doodle diagram, say $K$, such that the flat crossings are in $g(\cup_{i=1}^m N_i)$ and the virtual crossings are in the intersections of bands and annuli.  By definition, $(N(D), D)$ is $\varphi(K)$ and $(\Sigma, D)$ is a doodle diagram associated to $K$.  
This shows that $\Phi$ is surjective.  
We call such a $K$ a {\bf  virtual doodle diagram associated to $(\Sigma, D)$}.  

Before proving the injectivity of $\Phi$, we introduce the notion of Gauss data for doodle diagrams on surfaces, which is analogous to the notion of Gauss data for virtual doodle diagrams. 

Let $(\Sigma, D)$ and $(\Sigma', D')$ be doodle diagrams on surfaces with the same number of crossings.  
Let $\sigma$ be a bijection from the set of crossings of $D$ to that of $D'$.  
Let $N_1, \dots, N_m$ be regular neighbourhoods of crossings of $D$ and 
$N_1', \dots, N_m'$ be regular neighbourhoods of the corresponding crossings of $D'$ under $\sigma$.  
Let $f : \cup_{i=1}^m N_i \to  \cup_{i=1}^m N_i'$ be a homeomorphism such that 
for each $i$, $(N_i, D \cap N_i)$ is mapped to $(N_i', D' \cap N_i')$ homeomorphically with respect to their orientations.  Let $X= {\rm Cl}(\Sigma \setminus \cup_{i=1}^m N_i)$ and $X'= {\rm Cl}(\Sigma' \setminus \cup_{i=1}^m N_i')$.  The intersection $D \cap X$ is a union of some arcs and loops embedded in $X$.  
We say that $(\Sigma, D)$ and $(\Sigma', D')$ have the same {\bf  Gauss data} with respect to $\sigma$ if the homeomorphism $f: \cup_{i=1}^m N_i \to  \cup_{i=1}^m N_i'$ has an extension to 
a homeomorphism from $\cup_{i=1}^m N_i \cup (D \cap X)$ to $\cup_{i=1}^m N_i' \cup (D' \cap X')$.  

Note that doodle diagrams $(\Sigma, D)$ and $(\Sigma', D')$ have the same Gauss data with respect to a bijection between their crossings if and only if they are related by 
a finite sequence of homeomorphic equivalence and surface surgeries disjoint from diagrams.   

Now we prove that $\Phi$ is injective.  
Let $(\Sigma, D)$ and $(\Sigma', D')$ be doodle diagrams on surfaces and 
let  $K$ and $K'$ be their associated  virtual doodle diagrams.  

(1) Suppose that $(\Sigma, D)$ and $(\Sigma', D')$ have 
the same Gauss data with respect to a bijection between their crossings.  
Then $K$ and $K'$ have 
the same Gauss data with respect to a bijection between their flat crossings.  
By Lemma~\ref{lem:Gauss}, we see that $K$ and $K'$ represent the same  virtual doodle.  

(2) Suppose that $\Sigma =  \Sigma'$ and  
 $D'$ is obtained from $D$ by a move $H_1^{-}$ or $H_2^{-}$.  In the construction of $K$ and $K'$, 
taking embeddings $\tilde g: N(D) \to \R^3$ and $\tilde g' : N(D') \to \R^3$ suitably, we can obtain 
$K= g(D)$ and $K'= g'(D')$ such that $K'$ is obtained from $K$ by a move $R_1^{-}$ or $R_2^{-}$.  Thus $K$ and $K'$ represent the same  virtual doodle.   

(3) Suppose that $\Sigma =  \Sigma'$ and  
 $D'$ is obtained from $D$ by a move $H_1^{+}$ or $H_2^{+}$.  From the previous case, we see that 
$K$ and $K'$ represent the same  virtual doodle.   

Therefore we see that $\Phi$ is injective.  
\end{proof}

For an example of a doodle illustrated in two ways see Figures~\ref{fig15Kishino} and \ref{fig16Kishinosurface}.

\section{Minimal Virtual Doodles and Genera} \label{sect:minimalvirtual} 

In this section we will look at minimal virtual doodle diagrams and discuss about the genera of doodles.

\begin{definition}{\rm 
A virtual doodle diagram is {\bf minimal} if a doodle diagram associated to it is minimal.   
}\end{definition} 

Note that a virtual doodle diagram $K$ is minimal if and only if the doodle diagram 
obtained from $\varphi(K)= (F, D_F)$ by attaching discs along the boundary is a doodle diagram 
without monogons or bigons.  

For a virtual doodle, a minimal virtual doodle diagram is not unique.  However, we have the following. 

\begin{proposition}
Let $K$ and $K'$ be minimal oriented (or unoriented) virtual doodle diagrams representing the same oriented (or unoriented) virtual doodle.  
Then  $K$ and $K'$ are related by a finite sequence of detour moves modulo isotopies of $\R^2$. 
\end{proposition}

\begin{proof}
Let $(F, D)$ (or $(F', D')$) be the doodle diagram obtained from $\varphi(K)$ (or $\varphi(K')$) 
by attaching discs along the boundary.  Since $K$ and $K'$ are minimal and represent the same virtual doodle, 
$(F, D)$ and $(F', D')$ are minimal and represent the same doodle.  By Theorem~\ref{thm:minimal},  $(F, D)$ and $(F', D')$ are homeomorphic.  This implies that $K$ and $K'$ have the same Gauss data. By Lemma~\ref{lem:Gauss},  $K$ and $K'$ are related by a finite sequence of detour moves modulo isotopies of $\R^2$. 
\end{proof}

\begin{example}{\rm 
Figure~\ref{figd3_1} shows a minimal unoriented doodle diagram with $3$ crossings.  We denote it by $d3.1$. 

%$d3.1 \inlinediagram$\hfil\break
    \begin{figure}[h]
    \centerline{\epsfig{file=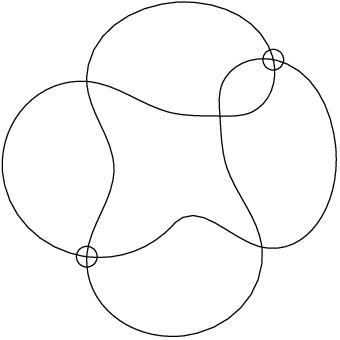, height=2.0cm}}
    \vspace*{8pt}
    \caption{$d3.1$}\label{figd3_1}
    \end{figure}
    
}\end{example}

\begin{example}{\rm 
Figure~\ref{figd4} shows $19$ minimal unoriented doodle diagrams with $4$ crossings. They are not equivalent each other as unoriented virtual doodles.  

    \begin{figure}[h]
    \centerline{\epsfig{file=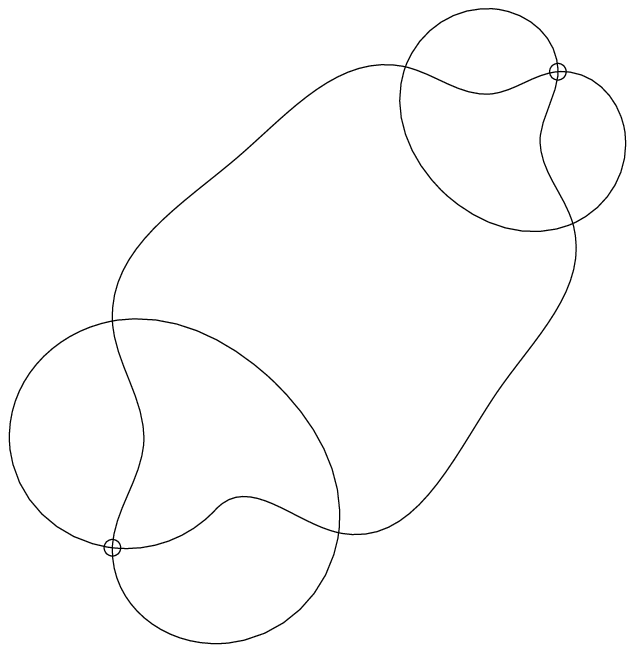, height=1.8cm}d4.1 \hfill \epsfig{file=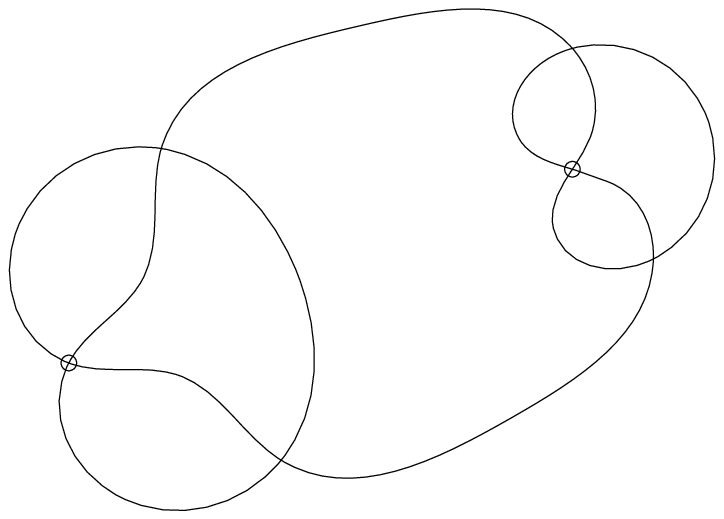, height=1.8cm}d4.2 \hfill \epsfig{file=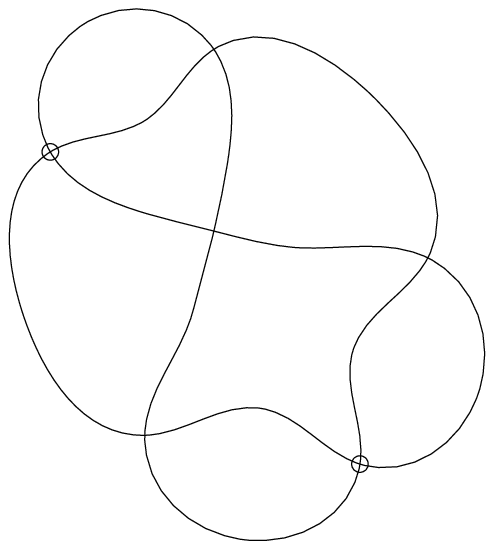, height=1.8cm}d4.3 \hfill \epsfig{file=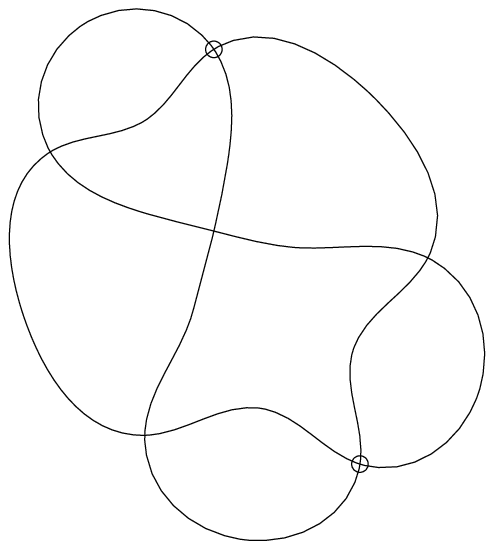, height=1.8cm}d4.4}
    \vspace*{8pt}
    \centerline{\epsfig{file=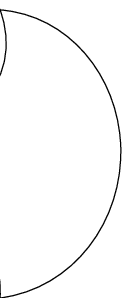, height=1.8cm}d4.5 \hfill \epsfig{file=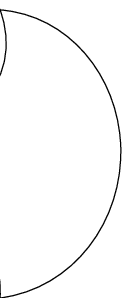, height=1.8cm}d4.6 \hfill  \epsfig{file=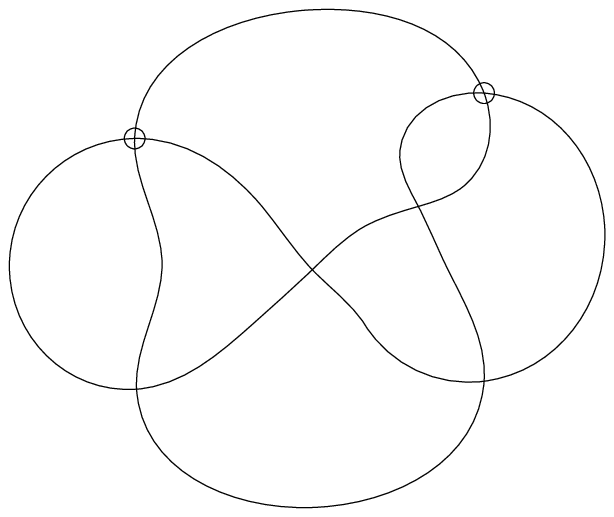, height=1.8cm}d4.7 \hfill \epsfig{file=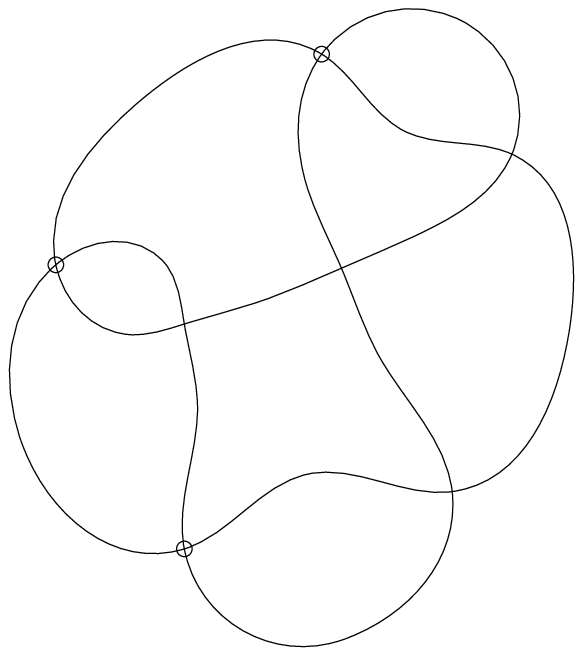, height=1.8cm}d4.8}
    \vspace*{8pt}
     \centerline{\epsfig{file=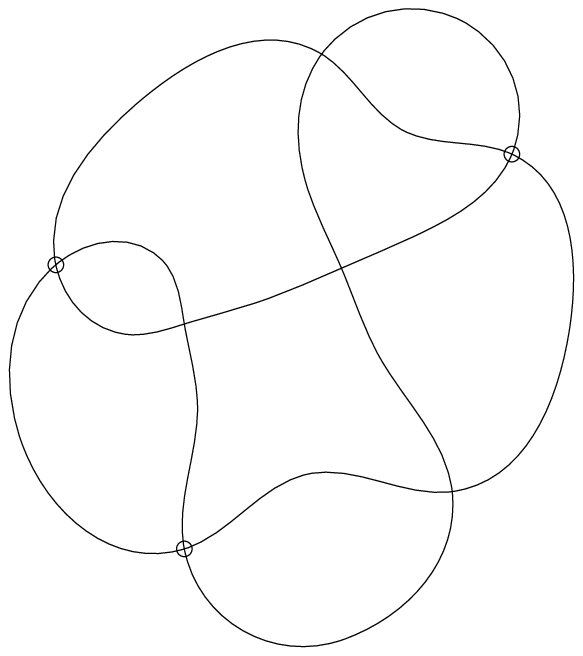, height=1.8cm}d4.9 \hfill \epsfig{file=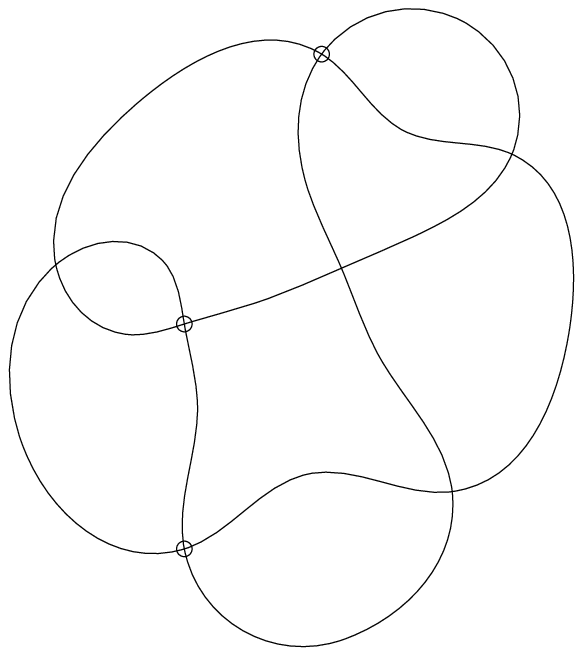, height=1.8cm}d4.10 \hfill \epsfig{file=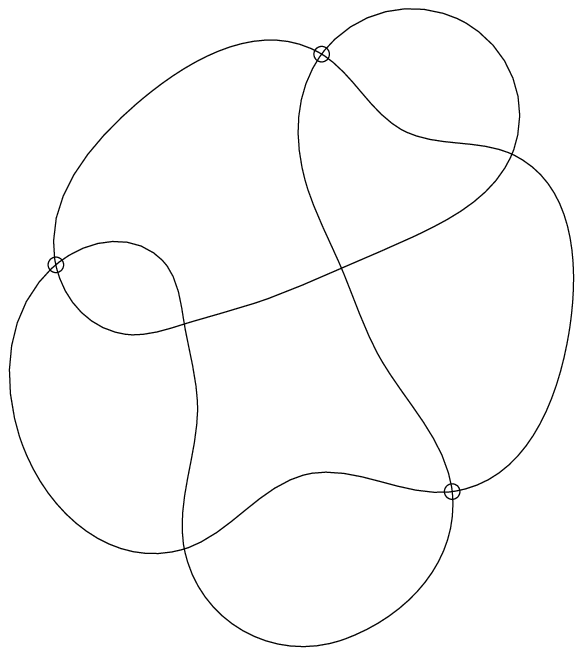, height=1.8cm}d4.11 \hfill \epsfig{file=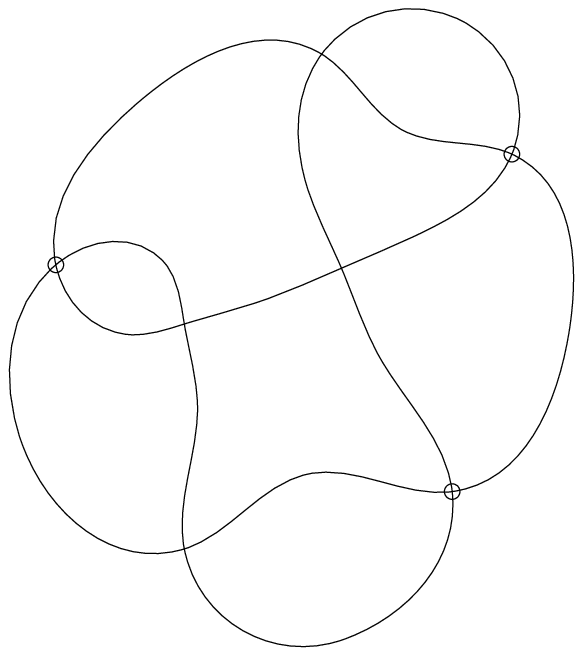, height=1.8cm}d4.12 }
    \vspace*{8pt}
     \centerline{\epsfig{file=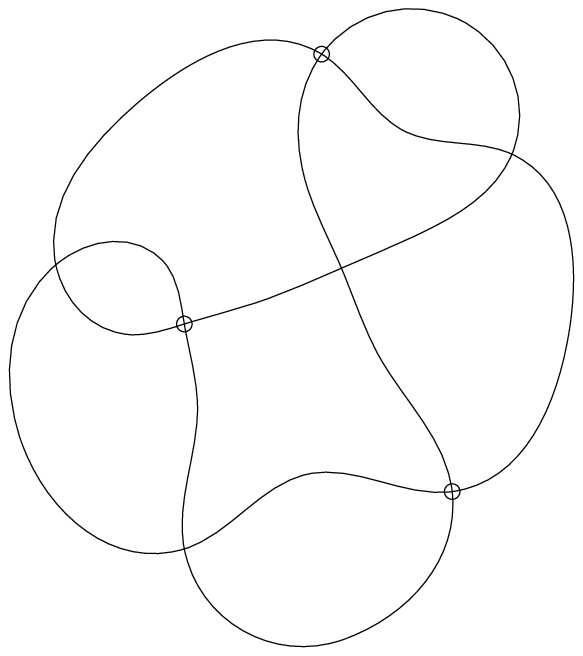, height=1.8cm}d4.13 \hfill \epsfig{file=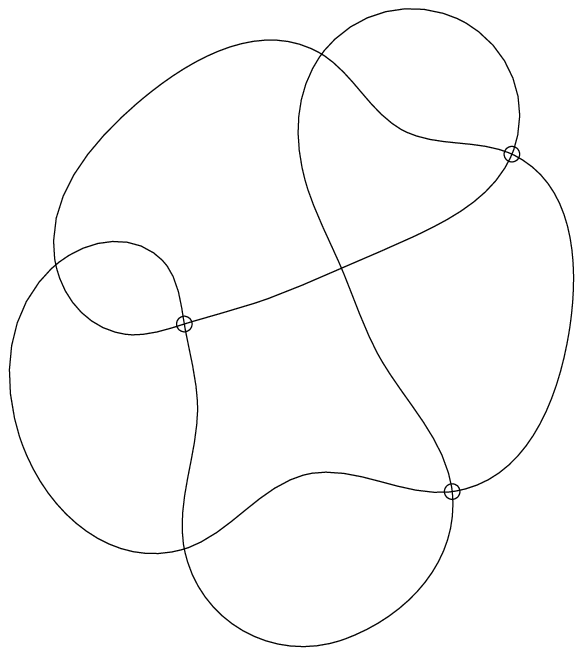, height=1.8cm}d4.14 \hfill \epsfig{file=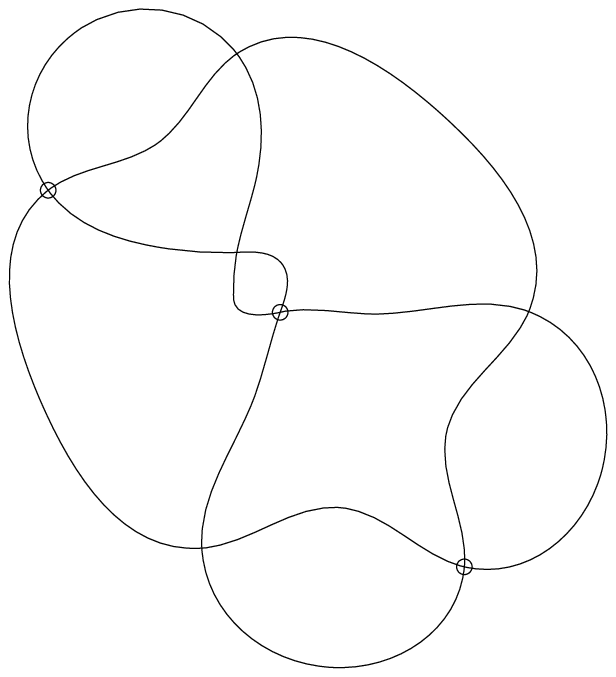, height=1.8cm}d4.15 \hfill \epsfig{file=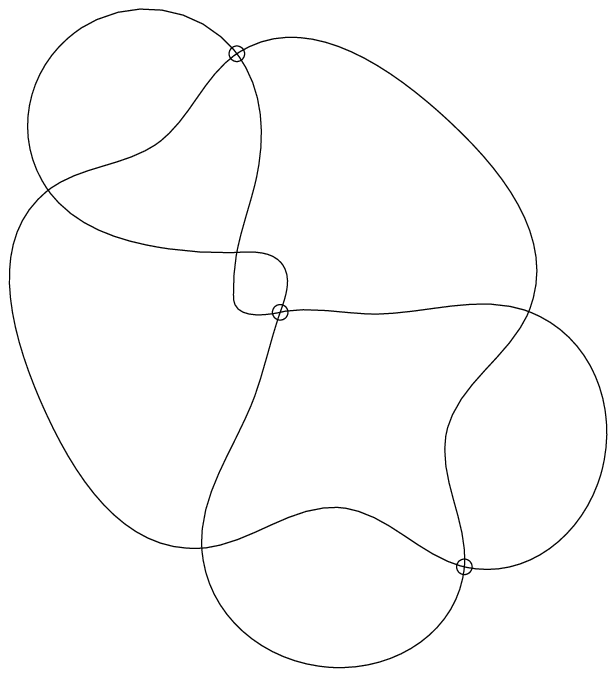, height=1.8cm}d4.16}
    \vspace*{8pt}
     \centerline{\epsfig{file=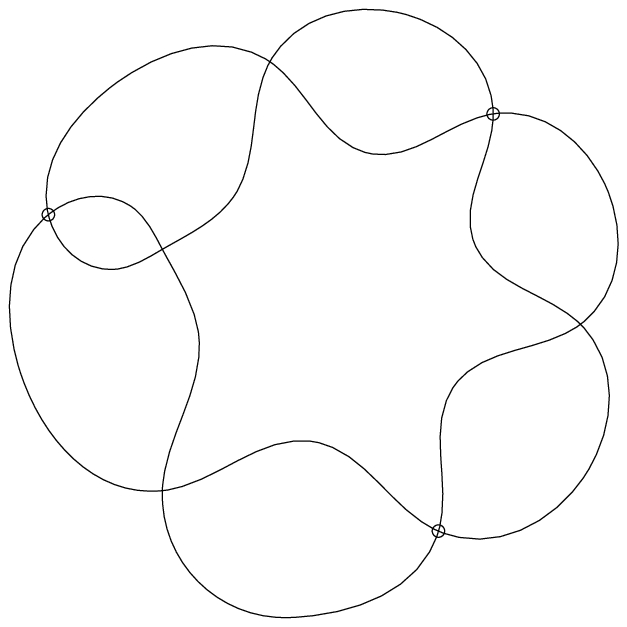, height=1.8cm}d4.17 \hfill \epsfig{file=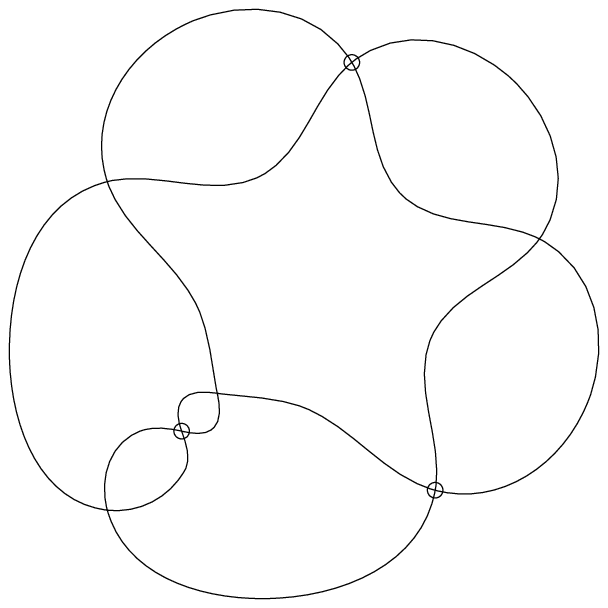, height=1.8cm}d4.18  \hfill \epsfig{file=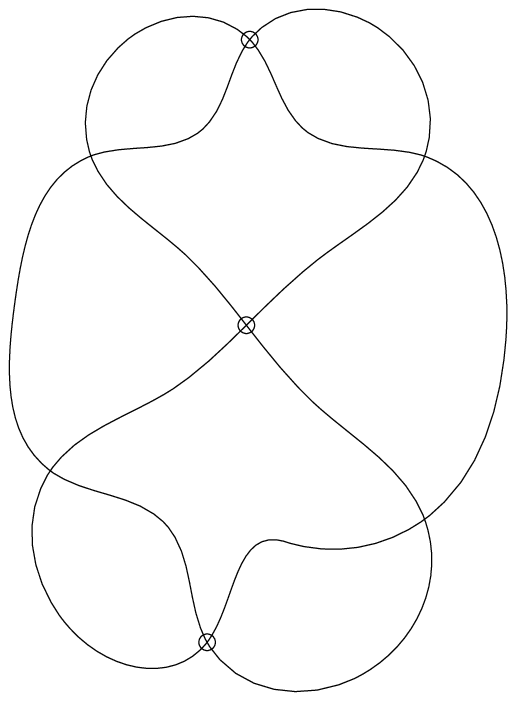, height=1.8cm}d4.19}
     \caption{Minimal Diagrams with $4$ Crossings}\label{figd4}
    \end{figure}

}\end{example}

\begin{remark}\label{Lebed} {\rm 
The diagrams $d4.18$ and $d4.19$ in the previous version, arXiv:1612.08473v1, are equivalent as unoriented virtual doodles. The authors would like to thank Victoria Lebed for pointing out this. 
}\end{remark}

\begin{example}{\rm 
See Figure~\ref{fig19minomalonecomponent}. It is a minimal unoriented virtual doodle diagram with $8$ crossings. 
Since the associated minimal doodle diagram is on the torus,  
by Corollary~\ref{cor:genus} the genus of the doodle is $1$.  Hence we cannot remove a virtual crossing.

%\diagram{A minimal doodle with one component and 8 crossings on a torus}
    \begin{figure}[h]
    \centerline{\epsfig{file=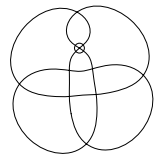, height=2.0cm}}
    \vspace*{8pt}
    \caption{A Minimal Virtual Doodle Diagram with One Component and $8$ Crossings}\label{fig19minomalonecomponent}
    \end{figure}

}\end{example}

It is easily seen that if $K$ has $m$ virtual crossings then 
the genus of the doodle $\Phi([K])$ associated to $K$ is  
equal to or less than $m$.   
Thus, the genus of the doodle is never greater than the minimum number of virtual crossings.  
It is quite easy to find examples where the inequality is strict. 

\begin{example}{\rm 
Let $K$ be the virtual doodle diagram depicted in Figure~\ref{fig20genus1doodle}. 
The doodle $\Phi([K])$ associated to $K$ has a minimal diagram on the torus and hence the genus of $\Phi([K])$ is $1$ by Corollary~\ref{cor:genus}.  
On the other hand, the minimum number of virtual crossings among all virtual doodle diagrams equivalent to $K$ is $2$, since the left circle must have at least one virtual crossing with each of the circles on the right.  
}\end{example} 

%\diagram{A genus 1 doodle with at least 2 virtual crossings}
    \begin{figure}[h]
    \centerline{\epsfig{file=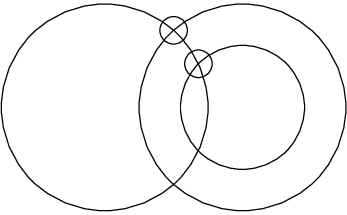, height=2.0cm}}
    \vspace*{8pt}
    \caption{A Genus $1$ Doodle with at Least $2$ Virtual Crossings}\label{fig20genus1doodle}
    \end{figure}

However it may be possible to amalgamate virtual crossings. If we take two consecutive virtual crossings on a virtual path then a surgery around the bases of the two handle representations means that the two handles can be replaced by one. If we apply this to the doodle in Figure~\ref{fig20genus1doodle}  we can represent it on a torus.

Indeed we can generalize this as follows. Define a {\bf virtual area} in a virtual diagram to be a square transverse to the diagram in which arcs enter in an edge and exit through the opposite edge with all the crossing points inside the square virtual. All these virtual crossings can be replaced by one handle. See Figure~\ref{fig21virtualarea}. 
 
%\diagram{Topological Interpretation of a Virtual Area}
    \begin{figure}[h]
    \centerline{\epsfig{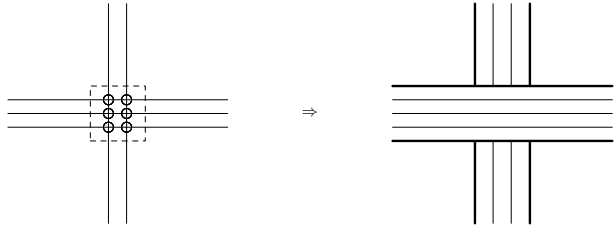}}
    \vspace*{8pt}
    \caption{Topological Interpretation of a Virtual Area}\label{fig21virtualarea}
    \end{figure}

Let $K$ be a  virtual doodle diagram.  A collection of virtual areas of $K$, say ${\cal A}= \{A_1, \ldots, A_k\}$,  
is a {\bf virtual area covering} of $K$ if they cover all virtual crossings and 
if they are mutually disjoint. 
Let ${\rm va} (K)$ denote the minimum cardinal number of all virtual area coverings of $K$.  When $K$ has no virtual crossings, we assume ${\rm va} (K)=0$.  We call ${\rm va} (K)$ the {\bf virtual area number} of $K$.  

For the virtual doodle $[K]$ represented by $K$, we define ${\rm va} ([K])$ by the minimum number among all ${\rm va} (K')$ such that $K'$ is equivalent to $K$.  We call it the {\bf virtual area number} of the  virtual doodle $[K]$.

Let $\Phi$ be the bijection from the family of  virtual doodles to the family of doodles defined in Section~\ref{sect:virtualsame}.

\begin{theorem}  Let $K$ be a  virtual doodle diagram and $D$ a doodle such that 
$\Phi ([K]) = [D]$.  Then the virtual area number ${\rm va} ([K])$ of the  virtual doodle $[K]$ equals the genus of the doodle $[D]$. 
\end{theorem}

\begin{proof}  First we show that $g([D]) \leq {\rm va} ([K])$.  Let $K_0$ be a  virtual doodle diagram with $[K_0]= [K]$ and ${\rm va}(K_0) = {\rm va} ([K])$.  Obviously, there is a 
doodle diagram $D_0$ on a closed surface $F_0$ 
with $[D_0] = \Phi([K_0])$ $(= \Phi([K]) = [D])$ 
and $g(F_0) = {\rm va}(K_0)$.  
Thus $g([D]) = g([D_0]) \leq g(F_0) = {\rm va}(K_0)= {\rm va}([K])$.  

We show that $g([D]) \geq {\rm va} ([K])$.  Let $D'$ be a doodle diagram on a closed surface $F'$ with $[D']=[D]$ and $g(F') = g([D])$. Let $g= g(F')$.  Consider a standard handle decomposition of $F'$, say $H^0 \cup  H^1_1 \cup \dots \cup H^1_{2g} \cup H^2$, where we assume that the attaching areas of the $1$-handles appear in the order 
$$ H^1_1, H^1_2, H^1_1, H^1_2, H^1_3, H^1_4, H^1_3, H^1_4, \dots$$ 
on the boudary of $H^0$. By moving $D'$ by an isotopy of $F'$, we may assume that (1) all crossings of $D'$ are in the interior of the $0$-handle $H^0$, (2)  for each $1$-handle $H^1_i$ the intersection of $D' \cap H^1_i$ is empty or 
some arcs that are parallel copies of the core of the $1$-handle, and (3) $D'$ is disjoint from the $2$-handle $H^2$.  Consider an immmersion of $H^0 \cup H^1_1 \cup \dots \cup H^1_{2g}$ to $\R^2$ such that the restriction of each handle is an embedding and the multiple point set consists of the transverse intersections of pairs of $1$-handles $H^1_{2i-1}$ and $H^1_{2i}$ for $i=1, \dots, g$. Then we have a  virtual doodle diagram, say $K'$, such that $\Phi([K'])=[D']$ and $K'$ has a virtual area covering whose cardinal number is less than or equal to $g$.  Thus, 
${\rm va}([K']) \leq g$.  Since $\Phi([K']) = [D'] = [D] = \Phi [K]$, we have $[K']=[K]$. 
Thus ${\rm va}([K])= {\rm va}([K']) \leq g= g(F') = g([D])$.  
\end{proof}

\begin{remark}{\rm 
Clearly a result, similar to the above result, can be proved for virtual links.
}\end{remark}

%\section{In the next paper}
%As advertised in the introduction the next paper will consider doodle braids, biquandle invariants, cobordism of doodles, the $\mu$ invariant, Khovanov group and commutator identities. We now give a taster menu for commutator identities.
%\subsection{Commutator Identities}
%In \cite{Fe} it is shown how a planar doodle gives rise to a collection of commutator identities. In fact there is a bijection between planar doodles with embedded components and simple commutator identities. For example consider the following examples from the borromean doodle.
%\diagram{}
%The left hand side yields the defining identity
%$$(a,c)^b(b,c)(b,a)^c(c,a)(c,b)^a(a,b)\equiv1$$
%where for example, $(a,c)^b=b^{-1}(a^{-1}c^{-1}ac)b$, and the right hand side yields
%$$(bc,a)(ca,b)(ab,c)\equiv1.$$
%Another which can be extracted from the borromean doodle is the Hall-Witt identity, \cite{H}. 
%$$((a,b),c^a)((c,a),b^c)((b,c),a^b)\equiv1$$
%This is a group-theoretic analogue of the Jacobi identity for Lie algebras.
%
%In the next paper we shall amplify this idea and relate it to basic commutators and the action of the braid group. Moreover we will extend the idea of commutator identities to doodles on any surface. Now the commutator words evaluate to elements of the fundamental group.


\begin{thebibliography}{0}

%\section{Bibliography}

\bibitem{A1}%Arnold, V. I.
V. I. Arnold,  
{\it Plane curves, their invariants, perestroikas and classifications. With
an appendix by F. Aicardi}, in Singularities and Bifurcations, Adv. 
Soviet Math., Vol. 21, pp. 33--91, Amer. Math. Soc.,
Providence, RI, 1994.

\bibitem{A2}%Arnold, V. I. 
V. I. Arnold,  
{\it Topological invariants of plane curves and caustics\/}. 
Dean Jacqueline B. Lewis Memorial Lectures presented at Rutgers
University, New Brunswick, New Jersey. University Lecture Series, 5.
Amer. Math. Soc., Providence, RI, 1994. 

\bibitem{BF}%Andrew Bartholomew and Roger Fenn, 
A. Bartholomew and R. Fenn, 
{\it Quaternionic invariants of virtual knots and links\/}, J. Knot Theory Ramifications 17 (2008), no. 2, 231--251.

\bibitem{B}%George M Bergman, 
G. M. Bergman, 
{\it The diamond lemma for ring theory\/}, Advances in Mathematics 29 (1978),  178--218. 

\bibitem{CKS} %J. Scott Carter, Seiichi Kamada and Masahico Saito,
J. S. Carter, S. Kamada and M. Saito, {\it Stable equivalence of knots on surfaces and virtual knot cobordisms\/},  J. Knot Theory Ramifications 11 (2002), no. 3, 311--322. 

\bibitem{C}%Peter Cromwell, 
P. Cromwell, {\it Polyhedra}, Cambridge University Press, Cambridge, 1997. 
    
\bibitem{FT}% Roger Fenn and Paul Taylor, 
R. Fenn and P. Taylor, 
{\it Introducing doodles}, Topology of low-dimensional manifolds (Proc. Second Sussex Conf., Chelwood Gate, 1977) Lecture Notes in Math., vol. 722, Springer, Berlin, 1979, pp. 37--43.

\bibitem{FTu} %Roger Fenn. and Turaev, V. (2007). 
R. Fenn and V. Turaev, 
{\it Weyl algebras and knots}, J. Geom. Phys. 57 (2007), 1313--1324.

\bibitem{Fe} % Roger Fenn, 
R. Fenn, 
{\it Techniques of geometric topology}, London Mathematical Society Lecture Note Series, vol. 57, Cambridge Univ. Press, Cambridge, 1983. 

\bibitem{FKR} %Roger Fenn, Ebru Keyman and Colin Rourke, 
R. Fenn, E. Keyman and C. Rourke, 
{\it The singular braid monoid embeds in a group\/},  
J. Knot Theory Ramifications 7 (1998), no. 7,  881--892.

\bibitem{F} %Roger Fenn, 
R. Fenn, 
{\it Generalised biquandles for generalised knot theories},  New Ideas in Low Dimensional Topology, pp. 79--103, 
Ser. Knots Everything 56, World Sci. Publ., Hackensack, NJ, 2015. 

\bibitem{H} %Marshall Hall 
M. Hall, {\it The Theory of Groups}, 
%Reprinting of the 1968 edition, 
Chelsea Publishing Co., New York, 1976. 
%Amer. Math. Soc., 1976. 

%\bibitem{IT} Noboru Ito and Yusuke Takimura, 
%{\it (1,2) and weak (1,3) homotopies on knot projections\/}, 
%J. Knot Theory Ramifications  22 (2013), no. 14, 1350085 (14 pp). 

\bibitem{Kad} % Teruhisa Kadokami
T. Kadokami, {\it Detecting non-triviality of virtual links\/}, 
J. Knot Theory Ramifications 12 (2003), no. 6, 781--803.  

\bibitem{KK} %Naoko Kamada and Seiichi Kamada, 
N. Kamada and S. Kamada, {\it Abstract link diagrams and virtual knots\/}, 
J. Knot Theory Ramifications 9 (2000), no. 1, 93--106.  

\bibitem{MK} %Mikhail Khovanov, 
M. Khovanov, {\it Doodle groups\/}, Trans. Amer. Math. Soc. 349 (1997), 2297--2315.  

\bibitem{K} % Louis H. Kauffman,
L. H. Kauffman,
{\it Virtual knot theory\/}, 
European Jour. Combinatorics  20 (1999), no. 7,  663--690. 

\bibitem{KS} % Toshimasa Kishino and Shin Satoh, 
T. Kishino and S. Satoh, 
{\it A note on classical knot polynomials\/},
J. Knot Theory Ramifications  13 (2004), no. 7, 845--856.

\bibitem{GK} %Greg Kuperberg, 
G. Kuperberg, {\it What is a virtual link?},  Algebr. Geom. Topol. 3 (2003), no. 20, 587--591. 


\bibitem{M} 
S. Matveev, {\it Algorithmic Topology and Classification of 3-Manifolds}, Algorithms Comput. Math., vol. 9,  Springer, Berlin, Germany, 2007. 

\bibitem{N} 
M. H. A. Newman. {\it On theories with a combinatorial definition of \lq\lq equivalence.\rq\rq} Ann. of Math (2), 43 (1942), 223--243.

%\bibitem{W} H. Whitney, {\it On regular closed curves in the plane\/}, 
%Compos. Math. 4 (1937), 276--284. 

\end{thebibliography}
\end{document}